\documentclass[10pt,a4paper,reqno]{amsart}
\usepackage{amsmath}
\usepackage{enumerate}
\usepackage{amsxtra}
\usepackage{amsopn}
\usepackage{bm}
\usepackage{amscd}
\usepackage{amssymb,amsthm}
\usepackage[colorlinks, linkcolor=black, anchorcolor=black, citecolor=black, hyperindex=false, urlcolor=black]{hyperref}
\usepackage{euscript}
\usepackage{mathrsfs}

\usepackage{amscd}
\usepackage{amsfonts}
\usepackage{latexsym}
\usepackage{verbatim}
\usepackage{color}
\usepackage{tikz-cd}

\theoremstyle{plain}

\newtheorem{thm}{Theorem}[section]
\newtheorem{cor}[thm]{Corollary}
\newtheorem{prop}[thm]{Proposition}
\newtheorem{lem}[thm]{Lemma}

\theoremstyle{definition}
\newtheorem{defn}[thm]{Definition}

\newtheorem{remark}[thm]{Remark}

\newtheorem*{claim*}{Claim}

\newcommand{\re}{\operatorname{Re}}
\newcommand{\im}{\operatorname{Im}}

\newcommand{\cc}{\mathbb{C}}

\renewcommand{\frac}[2]{\genfrac{}{}{0.6pt}{}{#1}{#2}}

\usepackage{titletoc}

\begin{document}

\title{Gromov hyperbolicity of pseudo-convex Levi corank one domains}
\author[B. Zhang]{Ben Zhang}

\address{Dept. of Mathematics, Northwest University, Xi'an, 710127, P.R. China}
\email{bzhang@nwu.edu.cn}

\subjclass[2010]{32F18, 32F45.}

\keywords{Kobayashi metric and distance, Gromov hyperbolicity, finite type.}


\begin{abstract}
After a study of the Kobayashi metrics on certain scaled domains, we show the stabilities of the infinitesimal Kobayashi metrics and the integrated distances in different scaling processes. As an application, we prove that bounded pseudo-convex domains of finite type where the Levi form of every boundary point has corank one are Gromov hyperbolic with respect to the Kobayashi metric. The results in this note generalize Zimmer's and Fiacchi's related works on Gromov hyperbolicity of weakly pseudo-convex domains.
\end{abstract}

\date{\today}
\maketitle{}


\section{Introduction}

In this note we study the Gromov hyperbolicity of a class of weakly pseudo-convex domains endowed with the Kobayashi metric. Loosely speaking, the Gromov hyperbolic metric spaces can be regarded as generalizations of negatively curved metric spaces and this notion has been used to study the extensions of the biholomorphisms and establish Denjoy-Wolff theorems on pseudo-convex domains.

The strongly pseudo-convex domains \cite{Balogh2000strongly}, complex ellipsoids in $\mathbb{C}^{d+1}$ \cite{gausser2018gromov}, smoothly bounded convex domains of finite type \cite{zimmer2016convexmathann}, and certain $\mathbb{C}$-proper weighted homogenous models \cite{zimmer2017Cconvex} are Gromov hyperbolic with respect to the Kobayashi metric. Recently, Fiacchi proved that smoothly bounded pseudo-convex domains of finite type in $\mathbb{C}^2$ are Gromov hyperbolic \cite{Fiacchi2020Gromov}.

In the present paper we shall explore the Gromov hyperbolicity of bounded pseudo-convex domains of finite type whose Levi form has corank one. These domains can be seen as generalizations of pseudo-convex of finite type domains in $\mathbb{C}^2$ and the complex ellipsoids appearing in \cite{gausser2018gromov}. Such domains have been studied by many authors, see \cite{bedford1991noncompactautomorphism}, \cite{ThaiandThu2009noncompact}, \cite{verma2015corankone} for instance. Motivated by Zimmer's work on convex domains of finite type \cite{zimmer2016convexmathann} and Fiacchi's strategy in \cite{Fiacchi2020Gromov}, we shall give the following theorem (here $d_\Omega$ is the Kobayashi distance):

\begin{thm}\label{thm0}
Let $\Omega\subset\mathbb{C}^{d+1}$ be a bounded pseudo-convex domain of finite type whose boundary is smooth of class $\mathcal{C}^{\infty}$, and suppose that the Levi form has rank at least $d-1$ at each point of the boundary. Then $(\Omega,d_\Omega)$ is Gromov hyperbolic.
\end{thm}

The main technique we used here is the scaling method (cf, \cite{kim2011book} and \cite{krantz2008scaling}), and the key step is to show the stabilities of the Kobayashi metrics and the integrated distances in different scaling processes. The paper is arranged as follows:

In Section~\ref{section2} we collect some basic facts about the Kobayashi metric and the Gromov hyperbolicity. In Section~\ref{section3} we recall local geometry near boundary points where the Levi form has corank one. In Section~\ref{section4} we take the scaling procedure and show the stability of the Kobayashi metrics. Particularly we give global estimates of the Kobayashi metrics on the limit domains. It should be noted here that the stability of the infinitesimal Kobayashi metrics does not always imply the stability of the integrated Kobayashi distances, thus in Section~\ref{section5}, we examine the stability of the Kobayashi distances under the scaling processes. The basic point is that the limit domains are complete hyperbolic and certain Kobayashi balls are compactly contained in the limit domains \cite{verma2015corankone}. In Section~\ref{section6}, we study the boundary properties of the geodesics (quasi geodesics) in the limit domains. In Section~\ref{section8}, we give the properties of Gromov products by utilizing the Catlin metric. As the limit domains are unbounded, in order to understand the asymptotic behaviors of the geodesics at infinity, we need to scale the limit domains again. This procedure is different from the scaling process in Section~\ref{section4}, thus we need to establish the stability theorems again, and this will be done in Section~\ref{section7}. In the final section, we shall give the proof of Theorem~\ref{thm0}.

\section{Preliminaries}\label{section2}

\subsection{Gromov hyperbolic spaces}
First we give some basic definitions of the Gromove hyperbolic spaces. Let $(X,d)$ be a metric space. A curve $\sigma\colon[a,b]\to X$ is a \textit{geodesic} if $d(\sigma(t_1),\sigma(t_2))=|t_1-t_2|$ for all $t_1,t_2\in[a,b]$. A \textit{geodesic triangle} in $(X,d)$ is a choice of three points in $X$ and geodesic segments connecting these points. A metric space is called \textit{proper} if the closed balls are compact.

\begin{defn}
A proper geodesic metric space $(X,d)$ is called $\delta$-\textit{hyperbolic} if every geodesic triangle is $\delta$-thin. If $(X,d)$ is $\delta$-hyperbolic for some $\delta\geq0$ then $(X,d)$ is called \textit{Gromov hyperbolic}.
\end{defn}

For every $x,y,o\in X$, the \textit{Gromov product} is defined by
\[
(x|y)_o:=\frac{1}{2}(d(x,o)+d(y,o)-d(x,y)).
\]
For $A\geq1, B\geq0$, a curve $\sigma\colon[a,b]\to X$ is an $(A,B)$-\textit{quasi geodesic} if
\[
A^{-1}|t_1-t_2|-B\leq d(\sigma(t_1),\sigma(t_2))\leq A|t_1-t_2|+B
\]
for all $t_1,t_2\in[a,b]$. For more materials about Gromov hyperbolicity, we refer the reader to \cite{Bridson1999book}.

\subsection{Kobayashi metrics and distances}Given a domain $\Omega\subset\mathbb{C}^{d+1}$, for $(p,X)\in\Omega\times\mathbb{C}^{d+1}$, the Kobayashi pseudo-metric is defined by
\[
F_{\Omega}(p;X)=\inf\big\{|\lambda| : f\colon\Delta\to\Omega\text{~~holomorphic~~}, f(0)=p,df_{|0}\lambda=X\big\}.
\]
If for all $p\in\Omega$ and $X\in\mathbb{C}^{d+1}$, $X\neq0$ implies $F_\Omega(p;X)>0$, we say that $\Omega$ is \textit{Kobayashi hyperbolic} and thus $F_\Omega$ is a metric. For $p,q\in\Omega$, define
\[
d_\Omega(p,q)=\inf_{\gamma}\bigg\{\int^1_0F_\Omega(\gamma(t);\gamma'(t))dt\bigg\}
\]
where the infimum takeing over all piecewise $\mathcal{C}^1$ curve $\gamma\colon[0,1]\to\Omega$ with $\gamma(0)=p$, $\gamma(1)=q$. If $F_\Omega$ is a metric, then $d_{\Omega}$ is a distance on $\Omega$. For more properties of the Kobayashi metric and distance, we refer the reader to \cite{kobayashi1998bookhyperbolic} or \cite{abate1990booktaut}.

\section{Special coordinates and polydiscs}\label{section3}

Let $\Omega\subset\mathbb{C}^{d+1}$ be a pseudo-convex domain with smooth defining function $r$. Let $u_\infty\in\partial\Omega$ be a point of finite type $2m$ in the sense of D'Angelo, and suppose the Levi form of $\partial\Omega$ has at least $d-1$ positive eigenvalues at $u_\infty$. We may assume that $|\partial r/\partial z_0(x)|\geq c>0$ for all $x$ in a neighborhood $W$ of $u_\infty$ and $\partial\overline{\partial}r(x)[L^r_i,\overline{L_j^r}]_{2\leq i,j\leq d}$ has $(d-1)$-positive eigenvalues in $W$ where
\[
L_0^r=\frac{\partial}{\partial z_0},\quad L_j^r=\frac{\partial}{\partial z_j}-\bigg(\frac{\partial r}{\partial z_0}\bigg)^{-1}\frac{\partial r}{\partial z_j}\frac{\partial }{\partial z_0},\quad j=1,\ldots,d.
\]
The set $\{L_1^r,\ldots,L_d^r\}$ forms a basis of the complex tangent space of $\partial\Omega$ near $u_\infty$. Moreover, we have
\[
\partial\overline{\partial} r(u_\infty)[L_i^r,\overline{L_j^r}\,]=\delta_{ij},\quad 2\leq i,j\leq d,
\]
where $\delta_{ij}=1$ if $i=j$ and $\delta_{ij}=0$ otherwise.

For any integers $j,k>0$, for $x\in W$, set
\[
\mathcal{L}_{j,k}^r\partial\overline{\partial} r(x)=\underbrace{L_1^r\ldots L_1^r}_{(j-1)\text{times}}\underbrace{\overline{L_1^r}\cdots\overline{L_1^r}}_{(k-1)\text{times}}
\partial\overline{\partial} r(x)[L_1^r,\overline{L_1^r}\,],
\]
and define
\[
C_\ell^r(x)=\max\{|\mathcal{L}_{j,k}^r\partial\overline{\partial}r(x)| : j+k=\ell\}.
\]
Then we have the following estimate of the Kobayashi metric \cite[Theorem~1.]{cho1995corankone}:
\begin{thm}\label{thm3}
Let $\Omega$ be a smoothly bounded pseudoconvex domain in $\mathbb{C}^{d+1}$ and let $u_\infty\in\partial\Omega$ be a point of finite type $2m$ in the sense of D'Angelo. Also assume that the Levi form $\partial\overline{\partial}r(z)$ of $\partial\Omega$ has at least $(d-1)$ positive eigenvalues at $u_\infty$. Then there exist a neighborhood $W$ about $u_\infty$ and a positive constant $A\geq1$ such that for $X=b_1L_0^r+b_1L_1^r+\cdots+b_dL_d^r$ at $z\in \Omega\cap W$,
\[
A^{-1}M_r(z;X)\leq F_\Omega(z;X)\leq AM_r(z;X),
\]
where
\[
M_r(z;X):=\frac{|b_0|}{|r(z)|}+|b_1|\sum^{2m}_{\ell=2}\bigg|\frac{C_\ell^r(z)}{r(z)}\bigg|^{1/\ell}+\sum^d_{k=2}
\frac{|b_k|}{\sqrt{|r(z)|}}.
\]
\end{thm}

Now we show that about each point $x\in W$, there is a special coordinate system and we use this coordinates to define certain ``distorted" polydiscs which are intrinsically attached near the boundary points. These polydiscs are closely related to the local geometry of the boundary. Under the hypotheses that $u_\infty\in\partial\Omega$, $|\partial r/\partial z_0(x)|\geq c>0$ for all $x\in W$, and $\partial\overline{\partial}r(x)[L^r_i,\overline{L_j^r}]_{2\leq i,j\leq d}$ has $(d-1)$-positive eigenvalues in $W$, we have the following:

\begin{prop}[{\cite[Proposition~2.2]{cho1994levirankn-2}}]
For each $x\in W$ and positive integer $2m$, there is a biholomorphism $\Phi_{x}\colon\mathbb{C}^{d+1}\to\mathbb{C}^{d+1}$, $\Phi_{x}(x)=0$ such that
\begin{align*}
r\circ\Phi_{x}^{-1}(\zeta)=&r(x)+\re(\zeta_0)+\sum^d_{\alpha=2}\sum_{\substack{j+k\leq m\\j,k>0}}\re\big(b^\alpha_{j,k}(x)\zeta^j_1\overline{\zeta}^k_1\zeta_\alpha\big)\notag\\
&+\sum_{\substack{j+k\leq 2m\\j,k>0}}a_{j,k}(x)\zeta_1^j\overline{\zeta}_1^k+\sum^d_{\alpha=2}|\zeta_\alpha|^2+R(x;\zeta)
\end{align*}
where $R(x;\zeta)$ is the error term.
\end{prop}

Moreover, the mappings $\Phi_x$ satisfies the condition
\begin{equation}\label{eq1}
\Phi_x(x)=0\quad\text{and}\quad\Phi_x(x+(-\epsilon,'0))=(-2\epsilon\partial r/\partial\overline{z}_0(x),'0)\in\mathbb{C}\times\mathbb{C}^d.
\end{equation}
If we set $r_x:=r\circ\Phi^{-1}_x$, let $P^x:=[P^x_{\alpha,j}]_{2\leq\alpha,j\leq d}$ be the unitary matrix such that
\[
{\overline{P}^x}^{\mathrm{T}}\bigg[\frac{\partial^2 r(x)}{\partial z_\alpha\partial\overline{z}_j}\bigg]_{2\leq\alpha,j\leq d}P^x=\begin{pmatrix}
\lambda_2^x\\
&\ddots&\\
&&\lambda_d^x
\end{pmatrix}
\]
where the eigenvalues $\lambda_j^x>0$, $j=2,\ldots,d$, then by \cite[Equation~(2.15)]{cho1994levirankn-2} (on page 808), we have
\begin{align}\label{eq8}
(d\Phi_x)L^r_0&=\frac{\partial}{\partial\zeta_0}=2\frac{\partial r}{\partial\overline{z}_0}(x)L_0^{r_x}\notag\\
(d\Phi_x)L^r_1&=\frac{\partial}{\partial\zeta_1}-\bigg(\frac{\partial r_x}{\partial\zeta_0}\bigg)^{-1}\bigg(\frac{\partial r_x}{\partial\zeta_1}\bigg)\frac{\partial}{\partial\zeta_0}=L^{r_x}_1\notag\\
\text{and for~~}\alpha&=2,\ldots,d\\
(d\Phi_x)L_\alpha^r&=\sum^d_{j=2}\overline{P}_{\alpha, j}^x(\lambda_j^x)^{-\frac{1}{2}}\bigg[\frac{\partial}{\partial\zeta_j}-\bigg(\frac{\partial r_x}{\partial\zeta_0}\bigg)^{-1}\bigg(\frac{\partial r_x}{\partial\zeta_j}\bigg)\frac{\partial}{\partial\zeta_0}\bigg]=\sum^d_{j=2}\overline{P}_{\alpha, j}^x(\lambda_j^x)^{-\frac{1}{2}}L^{r_x}_j.\notag
\end{align}

\begin{remark}
If we assume $u_\infty$ is the origin and take a linear transformation of $\mathbb{C}^{d+1}$ (which is clearly an automorphism of $\mathbb{C}^{d+1}$) beforehand such that
\[
r(z)=\re(z_0)+f(z_1,\ldots,z_d,\im(z_0))
\]
near $u_\infty$, then for $x\in W$, $\partial r/\partial \overline{z}_0(x)=1/2$ and $\Phi_x(x+(-\epsilon,'0))=(-\epsilon,'0)$.
\end{remark}

Now we define the polydiscs. For $x\in W$, set
\begin{align*}
A_\ell(x)&=\max\{|a_{j,k}(x)| : j+k=\ell\},\quad 2\leq\ell\leq 2m\\
B_{\ell'}(x)&=\max\{|b_{j,k}^\alpha(x)| : j+k=\ell',2\leq\alpha\leq d\},\quad2\leq\ell'\leq m
\end{align*}
and for each $\delta>0$, define
\[
\tau(x,\delta)=\min\{(\delta/A_\ell(x))^{1/\ell}, (\delta/B_{\ell'}(x))^{1/\ell'} : 2\leq\ell\leq 2m, 2\leq\ell'\leq m\}.
\]
If fact, it was shown that the coefficients $b^\alpha_{j,k}(x)$ are insignificant and may be dropped out, so
\[
\tau(x,\delta)=\min\{(\delta/A_\ell(x))^{1/\ell}: 2\leq\ell\leq 2m\}.
\]
Moreover, if we set
\[
\eta(x,\delta)=\min\{(\delta/C_\ell^r(x))^{1/\ell} : 2\leq\ell\leq 2m\},
\]
it follows that $\tau(x,\delta)\simeq\eta(x,\delta)$ for $x\in W$ and $\delta>0$ (see \cite[Propsition~2.7.]{cho1994levirankn-2}). Here $\simeq$ means equivalence up to a positive constant which only depends on $\Omega$.

If $u_\infty$ is of finite type $2m$, we have $A_{2m}(u_\infty)\geq 2c'>0$, and hence we may assume that $A_{2m}(x)\geq c'>0$ for all $x\in W$ by shrinking $W$ if necessary. Thus we have the estimate
\[
\delta^{1/2}\lesssim\tau(x,\delta)\lesssim\delta^{1/(2m)},\quad x\in W.
\]
Now set
\[
\tau_0(x,\delta)=\delta,\quad \tau_1(x,\delta)=\tau(x,\delta), \quad \tau_2(x,\delta)=\cdots=\tau_d(x,\delta)=\delta^{1/2},
\]
define
\[
R_\delta(x)=\{\zeta\in\mathbb{C}^{d+1} : |\zeta_k|<\tau_k(x,\delta), k=0,1,2,\ldots,d\},
\]
and
\[
Q_\delta(x)=\{\Phi_x^{-1}(\zeta) : \zeta\in R_\delta(x)\}.
\]

\begin{remark}\label{remark4}
It was shown (see \cite[on page 287]{verma2015corankone}) that there exist positive constants $0<\delta_e<1$ and $C_e>1$ such that, for all $x\in W$ and $0<\delta<\delta_e$,
\begin{enumerate}
\item these ``polydiscs" satisfy the engulfing property, i.e.,  if $y\in Q_\delta(x)$, then we have $Q_\delta(y)\subset Q_{C_e\delta}(x)$, and
\item if $y\in Q_\delta(x)$, then $\tau(y,\delta)\leq C_e\tau(x,\delta)\leq C_e^2\tau(x,\delta)$.
\end{enumerate}
\end{remark}

\section{Scaling processes and stability of the Kobayashi metrics}\label{section4}

In this section, we take shall a scaling process near a boundary point $u_\infty\in\partial\Omega$ and give the estimates of the Kobayashi metrics on the scaled domains. We begin with the following theorem:
\begin{thm}\label{thm5}
Let $\Omega\subset\mathbb{C}^{d+1}$ be a bounded pseudo-convex domain of finite type with smooth boundary. Suppose that Levi form has rank at least $d-1$ at each boundary point. Let $\{u_n\}\subset\Omega$ be a sequence that converges to a point $u_\infty\in\partial\Omega$. Then there exist a sequence $\{\psi_n\}\subset\mathrm{Aut}(\Omega)$ and a polynomial $P\colon\mathbb{C}\to\mathbb{R}$ subharmonic, without harmonic terms and $P(0)=0$, such that
\begin{enumerate}
\item[$\mathrm{(i)}$] there exists $\epsilon_n\to0^+$ such that $r_n:=\frac{1}{\epsilon_n}r\circ\psi^{-1}_n\to r_P$ locally uniformly on compact subsets of $\mathbb{C}^{d+1}$;
\item[$\mathrm{(ii)}$] $\psi_n(u_n)\to(-1,'0)\in\mathbb{C}\times\mathbb{C}^d$.
\end{enumerate}
\end{thm}

\begin{proof}
We may assume that $u_\infty$ is the origin and $\Omega$ is locally defined by a smooth function $r$ in a neighborhood $W$ of $u_\infty\in\partial\Omega$. Suppose $u_\infty$ is of finite type $2m$ and $|\partial r/\partial\overline{z}_0(u_\infty)|\geq c>0$, and $\partial\overline{\partial}r(x)[L^r_i,\overline{L_j^r}]_{2\leq i,j\leq d}$ has $(d-1)$-positive eigenvalues in $W$. Then for $n$ large enough, pick $\epsilon_n>0$ such that $\xi_n:=u_n+(\epsilon_n,0,\ldots,0)\in\partial\Omega$. It follows that $\epsilon_n\to0$ since $u_n\to u_\infty$ and consequently $\xi_n\to u_\infty$. Set
\[
\Phi_n(z):=\Phi_{\xi_n}(z)=(\zeta_0,\zeta_1,\ldots,\zeta_d),
\]
then $\Phi_n(u_n)=(-2\epsilon_n\partial r/\partial\overline{z}_0(\xi_n),0,\ldots,0)$ and $\partial r/\partial\overline{z}_0(\xi_n)\to 1/2$ when $n\to\infty$.

Next define a dilation
\[
\Lambda_n(\zeta_0,\zeta_1,\cdots,\zeta_d):=\bigg(\frac{\zeta_0}{\tau_0(\xi_n,\epsilon_n)},\frac{\zeta_1}{\tau_1(\xi_n,\epsilon_n)},\ldots,
\frac{\zeta_d}{\tau_d(\xi_n,\epsilon_n)}\bigg).
\]
Set $\psi_n:=\Lambda_n\circ\Phi_n$ and $\Omega_n:=\psi_n(\Omega)$, then
\begin{align*}
r_n:=&\frac{1}{\epsilon_n}r\circ\Phi_n^{-1}\circ\Lambda_n^{-1}(\zeta)\\
=&\re(\zeta_0)+\sum_{\substack{j+k\leq 2m\\j,k>0}}a_{j,k}(\xi_n)\epsilon_n^{-1}\tau(\xi_n,\epsilon_n)^{j+k}\zeta_1^j\overline{\zeta}_1^k+\sum^d_{\alpha=2}
|\zeta_\alpha|^2\\
&+\sum^d_{\alpha=2}\sum_{\substack{j+k\leq m\\j,k>0}}\re\big[\big(b^\alpha_{j,k}(\xi_n)\epsilon_n^{-\frac{1}{2}}\tau(\xi_n,\epsilon_n)^{j+k}
\zeta_1^j\overline{\zeta}_1^k\big)\zeta_\alpha\big]+
\mathrm{O}(\tau(\xi_n,\epsilon_n))
\end{align*}
is a local defining function of $\Omega_n$ near the origin. It was shown that for $\alpha=2,\ldots,d$ and $\epsilon_n>0$ small, $|\zeta_1|\leq1$, (\cite[Lemma~2.4]{cho1994levirankn-2})
\[
\bigg|\sum_{\substack{j+k\leq m\\j,k>0}}\re\big[b^\alpha_{j,k}(\xi_n)\epsilon_n^{-\frac{1}{2}}\tau(\xi_n,\epsilon_n)^{j+k}
\zeta_1^j\overline{\zeta}_1^k\big]\bigg|\leq\tau(\xi_n,\epsilon_n)^{\frac{1}{10}}\to0
\]
when $n\to\infty$. Thus
\[
r_n(\zeta)\to r_P(\zeta):=\re(\zeta_0)+P(\zeta_1,\overline{\zeta}_1)+\sum^d_{\alpha=2}|\zeta_\alpha|^2
\]
uniformly on compact subsets of $\mathbb{C}^{d+1}$. Here $P(\zeta_1,\overline{\zeta}_1)$ is a subharmonic polynomial of degree $2m'\leq2m$ without harmonic terms, and the Laplacian does not vanish identically, and we set the limit domain by
\[
\widehat{\Omega}:=\{ \zeta\in\mathbb{C}^{d+1} : r_P(\zeta)<0\}.
\]
\end{proof}

\begin{remark}\label{remark1}
As $\partial\widehat{\Omega}$ is also of finite type and the Levi form of every boundary point of $\widehat{\Omega}$ has at most corank one, every boundary point is a holomorphic peak point \cite{cho1996peakpoints}. Also \cite[Lemma~1.]{bedford1991noncompactautomorphism} shows that there is a holomorphic peak function at $\infty$ on $\widehat{\Omega}$, thus \cite[Theorem~1.]{gaussier1999tautness} implies that $\widehat{\Omega}$ is complete hyperbolic, hence taut.
\end{remark}

If we denote by $P^n_{\alpha,j}:=P^{\xi_n}_{\alpha,j}$ and $\lambda^{\xi_n}_j:=\lambda_j^{(n)}$, it follows from \eqref{eq8} that

\begin{lem}
For $p\in\psi_n(\Omega\cap W)$ and $X\in\mathbb{C}^{d+1}$, set $p^*:=\psi_n^{-1}(p)$, we have
\begin{align*}
(d\psi_n)L^r_0(p^*)&=\frac{d_n}{\epsilon_n}\frac{\partial}{\partial\zeta_0}=\frac{d_n}{\epsilon_n}L_0^{r_n}(p)\\
(d\psi_n)L^r_1(p^*)&=\frac{1}{\tau_1(\xi_n,\epsilon_n)}\bigg[\frac{\partial}{\partial\zeta_1}-\bigg(\frac{\partial r_n}{\partial\zeta_0}\bigg)^{-1}_{|p}\bigg(\frac{\partial r_n}{\partial\zeta_1}\bigg)_{|p}\frac{\partial}{\partial\zeta_0}\bigg]
=\frac{1}{\tau_1(\xi_n,\epsilon_n)}L^{r_n}_1(p)\\
\text{and for~~}\alpha&=2,\ldots,d\\
(d\psi_n)L_\alpha^r(p^*)&=\sum^d_{j=2}\frac{\overline{P}_{\alpha, j}^n(\lambda_j^{(n)})^{-\frac{1}{2}}}{\tau_j(\xi_n,\epsilon_n)}\bigg[\frac{\partial}{\partial\zeta_j}-\bigg(\frac{\partial r_n}{\partial\zeta_0}\bigg)^{-1}_{|p}\bigg(\frac{\partial r_n}{\partial\zeta_j}\bigg)_{|p}\frac{\partial}{\partial\zeta_0}\bigg]=\sum^d_{j=2}\frac{\overline{P}_{\alpha, j}^n(\lambda_j^{(n)})^{-\frac{1}{2}}}{\tau_j(\xi_n,\epsilon_n)}L^{r_n}_j(p)
\end{align*}
where $d_n=2\partial r/\partial z_0(\xi_n)$ and $d_n\to1$ when $n\to\infty$.
\end{lem}

\begin{thm}\label{thm2}
Notations as Theorem~\ref{thm5}. There exists $A_0\geq1$ such that for all $p\in\psi_n(\Omega\cap W)$ and $X\in\mathbb{C}^{d+1}$,
\[
A_0^{-1}M_{r_n}(p;X)\leq M_r(p^*;(d\psi_n^{-1})_{|p}X)\leq A_0M_{r_n}(p;X)
\]
where $p^*=\psi_n^{-1}(p)$ and $M_{r_n}\to M_{r_P}$ locally uniformly on $\widehat{\Omega}$.
\end{thm}

\begin{proof}
The proof is similar to \cite[Theorem~4.1]{Fiacchi2020Gromov}. By the previous lemma,
\begin{align*}
\frac{\tau(\xi_n,\epsilon_n)^2}{\epsilon_n}\mathcal{L}^r_{1,1}(p^*)&=\frac{\tau(\xi_n,\epsilon_n)^2}{\epsilon_n}
\partial\overline{\partial}r(p^*)\big[L^r_1,\overline{L_1^r}\big]\\
&=\frac{\tau(\xi_n,\epsilon_n)^2}{\epsilon_n}\partial\overline{\partial}(r\circ\psi_n^{-1})(p)\big[(d\psi_n)L^r_1,
(d\psi_n)\overline{L^r_1}\big]\\
&=\partial\overline{\partial}\bigg(\frac{1}{\epsilon_n}r\circ\psi_n^{-1}\bigg)(p)\big[L^{r_n}_1,
\overline{L^{r_n}_1}\,\big]=\mathcal{L}^{r_n}_{1,1}(p).
\end{align*}
Similarly, for $\ell\in\{2,\ldots,2m\}$,
\[
\frac{\tau(\xi_n,\epsilon_n)^\ell}{\epsilon_n}C^r_\ell(p^*)=C^{r_n}_\ell(p)
\]
and
\[
C^{r_n}_\ell(p)\to C^{r_P}_\ell(p)=:A^P_\ell(p),\quad n\to\infty
\]
uniformly on compact subsets of $\mathbb{C}^{d+1}$.

Now, for any $X\in\mathbb{C}^{d+1}$, if we write $(d\psi_n^{-1})_{|p}X=:y_0L^r_0+\cdots+y_dL^r_d$, then
\begin{align*}
X&=y_0(d\psi_n)L_0^r+\cdots+y_d(d\psi_n)L_d^r\\
&=\frac{y_0d_n}{\epsilon_n}L_0^{r_n}+\frac{y_1}{\tau(\xi_n,\epsilon_n)}L^{r_n}_1+\sum^d_{\alpha=2}y_\alpha
\sum^d_{j=2}\frac{\overline{P}_{\alpha, j}^n(\lambda_j^{(n)})^{-\frac{1}{2}}}{\sqrt{\epsilon_n}}L^{r_n}_j\\
&=:x_0L_0^{r_n}+x_1L^{r_n}_1+x_2L^{r_n}_2+\cdots+x_dL_d^{r_n}
\end{align*}
and it follows that
\[
y_0=\epsilon_n x_0/d_n,\quad y_1=\tau(\xi_n,\epsilon_n)x_1
\]
and
\[
\begin{pmatrix}
y_2\\
\vdots\\
y_d
\end{pmatrix}=\sqrt{\epsilon_n}
\begin{pmatrix}
(\lambda_2^{(n)})^{1/2}\\
&\ddots&\\
&&(\lambda_d^{(n)})^{1/2}
\end{pmatrix}\big[\,\overline{P}^n_{\alpha,j}]^{-1}
\begin{pmatrix}
x_2\\
\vdots\\
x_d
\end{pmatrix}.
\]

Recall that $\big[\,\overline{P}^n_{\alpha,j}]^{-1}$ is an $(d-1)\times(d-1)$ unitary matrix. Let $c_*>0$ be a constant such that $0<c_*^{-1}\leq(\lambda_j^x)^{1/2}\leq c_*$ for all $x\in W$. As $d_n\to1$ we can choose $A_0\geq1$ for all $n$ large such that
\begin{align*}
M_r(p^*;X)&=\frac{|y_0|}{|r(p^*)|}+|y_1|\sum^{2m}_{\ell=2}\bigg|\frac{C_\ell^r(p^*)}{r(p^*)}\bigg|^{\frac{1}{\ell}}+
\sum^{d}_{\alpha=2}\frac{|y_\alpha|}{\sqrt{|r(p^*)|}}\\
&\leq\frac{1}{d_n}\frac{|x_0|}{|r_n(p)|}+|x_1|\sum^{2m}_{\ell=2}
\bigg|\frac{C_\ell^{r_n}(p)}{r_n(p)}\bigg|^{\frac{1}{\ell}}+\sum^d_{j=2}\frac{c_*\sqrt{d}|x_j|}{\sqrt{|r_n(p)|}}\\
&\leq A_0M_{r_n}(p;X)
\end{align*}
and similarly we have the opposite inequality
\[
A^{-1}_0M_{r_n}(p;X)\leq M_r(p^*;d\psi_n^{-1}X)\leq A_0M_{r_n}(p;X)
\]
and the proof is complete.

\end{proof}

Now we need the following theorem \cite[Lemma~3.3]{verma2015corankone}:

\begin{thm}\label{thm4}
For $(z;X)\in\widehat{\Omega}\times\mathbb{C}^{d+1}$,
\[
\lim_{n\to\infty}F_{\Omega_n}(z;X)=F_{\widehat{\Omega}}(z;X).
\]
Moreover, the convergence is uniform on compact subsets of $\widehat{\Omega}\times\mathbb{C}^{d+1}$.
\end{thm}

By Theorem~\ref{thm3} and Theorem~\ref{thm2}, if fact, we have a global estimate of the Kobayashi metric on $\widehat{\Omega}$:
\begin{lem}\label{lem4}
There exists a constant $A\geq1$ such that for every $(z;X)\in\widehat{\Omega}\times\mathbb{C}^{d+1}$ where $X=(x_0,x_1,\ldots,x_n)$, if we denote by
\[
M_{r_P}(z;X):=\frac{|x_0+2x_1P'(z_1)+2\sum^d_{\alpha=2}x_\alpha\overline{z}_\alpha|}
{|r_P(z)|}+|x_1|\sum^{2m}_{\ell=2}\bigg|\frac{A^P_\ell(z_1)}
{r_P(z)}\bigg|^{\frac{1}{\ell}}+\sum^d_{\alpha=2}\frac{|x_\alpha|}{\sqrt{|r_P(z)|}},
\]
where $A^P_{\ell}(z_1)=\max\big\{\big|\frac{\partial^{j+k}P}{\partial z_1^j\partial\overline{z}_1^k}(z_1)\big| :j,k>0,j+k=\ell\big\}$, then
\[
A^{-1}M_{r_P}(z;X)\leq F_{\widehat{\Omega}}(z;X)\leq AM_{r_P}(z;X).
\]
\end{lem}

\begin{proof}
If write $X=\sum^d_{k=0}b_k(X)L^{r_P}_k$ and $P'(z_1)=\partial P(z_1,\overline{z}_1)/\partial z_1$, it follows that
\begin{align*}
b_1(X)&=x_1,\ldots,b_d(X)=x_d\\
b_0(X)&=x_0+2P'(z_1)x_1+2\overline{z}_2x_2+\cdots+2\overline{z}_dx_d,
\end{align*}
and $M_{r_n}$ converges uniformly on compact subset to $M_{r_P}$.

Now for $(z,X)\in\widehat{\Omega}\times\mathbb{C}^{d+1}$, we have
\[
F_{\Omega_n}(z;X)=F_{\Omega}(\psi^{-1}_n(z);(d\psi_n^{-1})_{|z}X),
\]
and for all $n$ large, we may assume that $\psi_n^{-1}(z)\in\Omega\cap W$. Combine Theorem~\ref{thm3} we have
\[
A^{-1}M_r(\psi_n^{-1}(z);(d\psi_n^{-1})_{|z}X)\leq F_{\Omega}(\psi_n(z);(d\psi_n^{-1})_{|z}X)\leq AM_r(\psi_n^{-1}(z);(d\psi_n^{-1})_{|z}X)
\]
and follows by Theorem~\ref{thm2} that
\[
A^{-1}A_0^{-1}M_{r_n}(z;X)\leq F_{\Omega_n}(z;X)\leq AA_0M_{r_n}(z;X).
\]
Let $n\to\infty$, we have
\[
A^{-1}A_0^{-1}M_{r_P}(z;X)\leq F_{\widehat{\Omega}}(z;X)\leq AA_0M_{r_P}(z;X).
\]
For simplicity of notation we replace $AA_0$ by $A$ and the proof is complete.
\end{proof}

\begin{remark}
It should be noted here that, for a generic domain of the form
\[
D=\bigg\{z\in\mathbb{C}^{d+1} : r_\phi(z)=\re(z_0)+\phi(z_1,\overline{z}_1)+\sum^d_{\alpha=2}|z_\alpha|^2\bigg\},
\]
where $\phi$ is a subharmonic polynomial of even degree which contains no harmonic terms, we can always define a quantity $M_{r_\phi}(z;X)$, which is also a Finsler metric on $D$. But it is not clear whether $M_{r_\phi}$ is comparable with the Kobayashi metric $F_D$ on $D$ (if $\phi$ is homogeneous, they are comparable, see \cite[Proposition~3.3]{Fiacchi2020Gromov} or \cite[Theorem~1.]{herbort1992metrichomogeneous}).
\end{remark}

\section{Stability of the Kobayashi distances}\label{section5}

In this section, we shall investigate the Kobayashi distances on the scaled domains. The technique used here is a modification of \cite[Proposition~3.6]{verma2015corankone}. The main theorem is the following:

\begin{thm}\label{thm1}
Notations as in Theorem~\ref{thm5}, for any $p,q\in\widehat{\Omega}$,
\begin{equation}\label{eq0}
\lim_{n\to\infty}d_{\Omega_n}(p,q)=d_{\widehat{\Omega}}(p,q).
\end{equation}
Moreover, the convergence is uniform on compact subsets of $\widehat{\Omega}\times\widehat{\Omega}$.
\end{thm}

The proof is divided into several lemmas, first we have the following:

\begin{lem}\label{lem2}
For any $p,q\in\widehat{\Omega}$,
\begin{equation*}
\limsup_{n\to\infty}d_{\Omega_n}(p,q)\leq d_{\widehat{\Omega}}(p,q),
\end{equation*}
and the convergence is uniform on compact subsets of $\widehat{\Omega}\times\widehat{\Omega}$.
\end{lem}

\begin{proof}
Let $K\Subset\widehat{\Omega}$ be a compact subset. Thus, there is a uniform constant $R'>0$ such that $d_{\widehat{\Omega}}(p,q)\leq R'$ for any $p,q\in K$. Let $\gamma\colon[0,1]\to\widehat{\Omega}$ be a piecewise $\mathcal{C}^1$-curve such that $\gamma(0)=p$, $\gamma(1)=q$ and
\[
\int^1_0F_{\widehat{\Omega}}(\gamma(t);\gamma'(t))dt\leq d_{\widehat{\Omega}}(p,q)+\epsilon/2.
\]
As $K$ is compact, there exists a compact subset $K'\Subset\widehat{\Omega}$ which only depends on $K$ and $\epsilon$ such that $K\subset K'$ and $\gamma\subset K'$ for all $p,q\in K$. (Here the existence of $K'$ is ensured by the uniform constant $R'$ and the complete hyperbolicity of $\widehat{\Omega}$.)

Now by the uniform convergence of the Kobayashi metric on $K'$ (Theorem~\ref{thm4}), for all $n$ sufficiently large, one gets
\begin{align*}
d_{\Omega_n}(p,q)&\leq\int^1_0F_{\Omega_n}(\gamma(t);\gamma'(t))dt\\
&\leq\int^1_0F_{\widehat{\Omega}}(\gamma(t);\gamma'(t))dt+\epsilon/2\\
&\leq d_{\widehat{\Omega}}(p,q)+\epsilon.
\end{align*}
Thus
\[
\limsup_{n\to\infty}d_{\Omega_n}(p,q)\leq d_{\widehat{\Omega}}(p,q)+\epsilon
\]
and the proof is complete.
\end{proof}

Next we show the opposite inequality. Let
\[
v_n:=(-2\partial r/\partial\overline{z}_0(\xi_n),'0)\in\mathbb{C}\times\mathbb{C}^d\quad\text{and thus}\quad \psi_n(u_n)=v_n\to(-1,'0)
\]
in Theorem~\ref{thm5}. For an arbitrary hyperbolic domain $D\subset\mathbb{C}^{d+1}$, let $d_D$ be the Kobayashi distance on $D$. Denote by
\[
\mathbf{B}_D(x,\delta):=\big\{p\in D : d_D(x,p)<\delta\big\}
\]
the Kobayashi ball centered at $x\in D$ with radius $\delta>0$.

\begin{remark}\label{remark2}
By the above lemma, for each $\nu>0$, there exists a $R_\nu>0$ such that
\[
\mathbf{B}_{\widehat{\Omega}}((-1,'0),\nu)\subset \mathbf{B}_{\Omega_n}((-1,'0),R_\nu)
\]
for all $n>N$ large. In fact, for $x\in\mathbf{B}_{\widehat{\Omega}}((-1,'0),\nu)$, the previous lemma implies that there exists $N>0$ such that
\[
d_{\Omega_n}((-1,'0),x)\leq d_{\widehat{\Omega}}((-1,'0),x)+1/2< \nu+1:=R_\nu
\]
for all $n>N$ and consequently $x\in\mathbf{B}_{\Omega_n}((-1,'0),R_\nu)$.
\end{remark}

\begin{lem}[{\cite[Lemma~3.4.]{verma2015corankone}}]\label{lem3}
For any $R>0$ fixed, $\mathbf{B}_{\Omega_n}(v_n,R)$ is uniformly compactly contained in $\widehat{\Omega}$  for all $n$ large.
\end{lem}

\begin{remark}
This lemma relies on the engulfing property of the polydiscs in Remark~\ref{remark4}, and it is not clear whether it holds when $\Omega$ is unbounded and $v_n\to\infty$.
\end{remark}

\begin{lem}\label{lem7}
For any compact subset $K\Subset\widehat{\Omega}$, there exists $C_K>0$ which only depends on $K$, such that for all $R>0$ fixed, there exists $N>0$ such that
\[
\mathbf{B}_{\Omega_n}(\varpi,R)\subset\mathbf{B}_{\Omega_n}((-1,'0),C_K+1+R),\quad\varpi\in K
\]
for all $n>N$.
\end{lem}

\begin{proof}
First assume that $K\subset\Omega_n$ for all $n$ large since $\Omega_n\to\widehat{\Omega}$. Set $C_K:=\max\{d_{\widehat{\Omega}}((-1,'0),y) : y\in K\}$, by Lemma~\ref{lem2}, there exist $N>0$ such that
\[
d_{\Omega_n}((-1,'0),\varpi)\leq d_{\widehat{\Omega}}((-1,'0),\varpi)+1/2\leq C_K+1/2,\quad \varpi\in K
\]
for all $n>N$ since $K$ is compact.

Thus for any $R>0$ and any $x\in\mathbf{B}_{\Omega_n}(\varpi,R)$, we have
\[
d_{\Omega_n}((-1,'0),x)\leq d_{\Omega_n}((-1,'0),\varpi)+d_{\Omega_n}(\varpi,x)< C_K+1+R.
\]
for all $n>N$, which indicates
\[
\mathbf{B}_{\Omega_n}(\varpi,R)\subset\mathbf{B}_{\Omega_n}((-1,'0),C_K+1+R)
\]
uniformly for all $n>N$ and $\varpi\in K$ and the proof is complete.
\end{proof}

Now we need the following lemma (see for instance \cite[Lemma~2.1.]{krantz2009kobayashibunwong}):

\begin{lem}\label{lem1}
Let $D$ be a Kobayashi hyperbolic domain in $\cc^n$ with a sub-domain $D'\subset D$. Let $p,q\in D'$, $d_D(p,q)=a$ and $b>a$. If $D'$ satisfies the condition $\mathbf{B}_{D}(q,b)\subset D'$, then the following inequality holds:
\begin{equation*}
d_{D'}(p,q)\leq\frac{1}{\tanh(b-a)}d_D(p,q).
\end{equation*}
\end{lem}

\begin{proof}[Proof of Theorem~\ref{thm1}]
Let $K\Subset\widehat{\Omega}$ such that $(-1,'0)\in K$. As $\widehat{\Omega}$ is complete hyperbolic, it follows that
\[
\widehat{\Omega}=\bigcup_{\nu=1}\mathbf{B}_{\widehat{\Omega}}((-1,'0),\nu).
\]
As $K$ is compact, there exists a $\nu_0$ such that $K\subset\mathbf{B}_{\widehat{\Omega}}((-1,'0),\nu_0)$ and Remark~\ref{remark2} implies that $\mathbf{B}_{\widehat{\Omega}}((-1,'0),\nu_0)\subset\mathbf{B}_{\Omega_n}((-1,'0),R_0)$ for some $R_0>0$ for all $n$ large. Thus we have
\[
K\subset\mathbf{B}_{\widehat{\Omega}}((-1,'0),\nu_0)\subset\mathbf{B}_{\Omega_n}((-1,'0),R_0)
\]
for all $n$ large.

Next, for any $p,q\in K$, we have
\[
d_{\Omega_n}(p,q)\leq d_{\Omega_n}(p,(-1,'0))+d_{\Omega_n}((-1,'0),q)\leq 2R_0.
\]
For small $\epsilon>0$, let $R'$ be a number such that $R'\gg \max\{4R_0,2(C_K+1)\}$ and
\[
\tanh(R'-2R_0)\geq 1-\epsilon,
\]
then Lemma~\ref{lem7} reveals that
\begin{equation*}
\mathbf{B}_{\Omega_n}(p,R')\subset \mathbf{B}_{\Omega_n}((-1,'0),C_K+1+R')\subset \mathbf{B}_{\Omega_n}((-1,'0),2R')
\end{equation*}
uniformly for $p\in K$ for all $n$ large. By Lemma~\ref{lem1} ($\Omega_n$ and $\mathbf{B}_{\Omega_n}((-1,'0),2R')$ play the role of $D$ and $D'$, and let $b=R'$, respectively), we have
\begin{equation*}
d_{\mathbf{B}_{\Omega_n}((-1,'0),2R')}(p,q)\leq\frac{d_{\Omega_n}(p,q)}{\tanh(R'-d_{\Omega_n}(p,q))}
\leq\frac{d_{\Omega_n}(p,q)}{\tanh(R'-2R_0)}\leq\frac{d_{\Omega_n}(p,q)}{1-\epsilon}.
\end{equation*}

Also there exists $N>0$ such that
\[
\mathbf{B}_{\Omega_n}((-1,'0),2R')\subset \mathbf{B}_{\Omega_n}(v_n,2R'+1)\subset\widehat{\Omega}
\]
for all $n>N$ large, where the second inclusion comes from Lemma~\ref{lem3}. In fact, we can pick $\delta>0$ such that $\mathbb{B}((-1,'0),2\delta)\subset\Omega_n$ and $v_n\in\mathbb{B}((-1,'0),\delta)$ for all $n>N$ for some integer $N$, where $\mathbb{B}((-1,'0),2\delta)$ is the Euclidean ball centered at $(-1,'0)$ with radius $2\delta$. Thus for any $x\in\mathbf{B}_{\Omega_n}((-1,'0),2R')$, the decreasing property of the Kobayashi distance reveals that
\begin{align*}
d_{\Omega_n}(x,v_n)&\leq d_{\Omega_n}(x,(-1,'0))+d_{\Omega_n}((-1,'0),v_n)\\
&\leq 2R'+d_{\mathbb{B}((-1,'0),2\delta)}((-1,'0),v_n)<2R'+1
\end{align*}
for all $n>N$.

Consequently we have
\[
d_{\widehat{\Omega}}(p,q)\leq d_{\mathbf{B}_{\Omega_n}((-1,'0),2R')}(p,q)\leq\frac{d_{\Omega_n}(p,q)}{1-\epsilon}.
\]
Let $n\to\infty$ we have
\[
d_{\widehat{\Omega}}(p,q)\leq\frac{1}{1-\epsilon}\liminf_{n\to\infty}d_{\Omega_n}(p,q).
\]
Now Lemma~\ref{lem2} and the fact that $\epsilon>0$ is arbitrary complete the proof.
\end{proof}


\section{Geodesics in the limit domains}\label{section6}

In this section, we will study the boundary behaviors of the Kobayashi $(A,0)$ quasi geodesics in the limit domain $\widehat{\Omega}$, where $\widehat{\Omega}$ is obtained by $u_\infty\in\partial\Omega$, $\{u_n\}\subset\Omega$ in Theorem~\ref{thm5}. First we have the following lemma:
\begin{lem}\label{lem5}
For each $x=(x_0,x_1,\ldots,x_n)\in\partial\widehat{\Omega}$ and $a>0$, the curve
\[
\sigma(t):=x-(ae^{-t},0,\ldots,0),\quad t\in\mathbb{R}
\]
is an $(A,0)$ quasi geodesic with respect to the Kobayashi metric.
\end{lem}

\begin{proof}
For any $p,q\in\widehat{\Omega}$, let $\gamma\colon[0,1]\to\widehat{\Omega}$ be a generic piecewise $\mathcal{C}^1$ path connecting $p$ and $q$. By Lemma~\ref{lem4} we have
\begin{align}\label{eq7}
d_{\widehat{\Omega}}(p,q)&\geq A^{-1}\inf_\gamma\int^1_0 M_{r_P}(\gamma(s);\gamma'(s))ds\notag\\
&\geq A^{-1}\inf_\gamma\int^1_0\frac{|\gamma_0'(s)+2\gamma_1'(s)P'(\gamma_1(s))+2\sum^d_{\alpha=2}\gamma_\alpha'(s)
\overline{\gamma_\alpha(s)}|}{-\re(\gamma_0(s))-P(\gamma_1(s))-\sum^d_{\alpha=2}|\gamma_\alpha(s)|^2}ds\notag\\
&\geq A^{-1}\inf_\gamma\int^1_0\bigg|\frac{[r_P(\gamma(s))]'}{-r_P(\gamma(s))}\bigg|ds=A^{-1}\bigg|\log\bigg(\frac{r_P(p)}{r_P(q)}
\bigg)\bigg|.
\end{align}
Thus for $t_1<t_2$, we have $d_{\widehat{\Omega}}(\gamma(t_2),\gamma(t_1))\geq A^{-1}(t_2-t_1)$. Also a direction computation shows that
\[
\mathrm{length}(\sigma_{|[t_1,t_2]})\leq A\int^{t_2}_{t_1}M_{r_P}(\gamma(s);\gamma'(s))ds=A(t_2-t_1)
\]
and it follows that $\sigma$ is an $(A,0)$ quasi geodesic line.
\end{proof}

Now we need the uniform estimates of the Kobayashi metrics:

\begin{lem}\label{lem6}
Let $\Omega$, $\psi_n$, $P$ and $\widehat{\Omega}$ be as in Theorem~\ref{thm5}. Then for each $R>0$, there exist $c>0$ and $C>0$ such that for all $n$ large
\[
M_{r_n}(z;X)\geq\frac{c\|X\|}{|r_n(z)|^{1/{(2m)}}},\quad z\in\psi_n(\Omega)\cap \mathbb{B}(0,R),\quad X\in\mathbb{C}^{d+1}
\]
and for any $o\in\widehat{\Omega}$
\[
d_{\Omega_n}(z,o)\leq C+A\ln\bigg(\frac{1}{|r_n(z)|}\bigg),\quad z\in \psi_n(\Omega)\cap \mathbb{B}(0,R).
\]
Here $\mathbb{B}(0,R)$ is the Euclidean ball centered at the origin with radius $R$.
\end{lem}

\begin{proof}
For the metric estimates, only need to consider the inequality ($x_k,y,a_k\in\mathbb{C}$, and $\sum^d_{k=1}|a_k|\leq |a|$)
\begin{align*}
\bigg|y+\sum^d_{k=1}a_kx_k\bigg|+\sum^d_{k=1}|x_k|&\geq\frac{1}{1+\sum^d_{k=1}|a_k|}
\bigg(|y|-\sum^d_{k=1}|a_k||x_k|\bigg)+
\sum^d_{k=1}\frac{1+|a_k|}{1+|a|}|x_k|\\
&\geq\frac{1}{1+|a|}\bigg(\sum^d_{k=1}|x_k|+|y|\bigg).
\end{align*}
For the distance estimates, by Lemma~\ref{lem5} the proof is the same as Lemma~\cite[Lemma~5.5]{Fiacchi2020Gromov} and we omit it.
\end{proof}

\begin{cor}\label{cor1}
Let $\Omega$, $\psi_n$, $P$ and $\widehat{\Omega}$ be as in Theorem~\ref{thm5}. Then for each $R>0$ and $A\geq1$, there exists a $L:=L(R,A)>0$ such that, for every $n\in\mathbb{N}$, an $(A,0)$ quasi geodesic $\sigma$ of $\psi_n(\Omega)$ contained in $B_R(0)$ is $L$-Lipschitz (with respect to the Euclidean distance of $\mathbb{C}^{d+1}$) and
\[
M_{r_n}(\sigma(t);\sigma'(t))\leq A
\]
for almost every $t\in[0,T]$.
\end{cor}

\begin{proof}
By the previous lemma the proof is the same as \cite[Corollary~5.6]{Fiacchi2020Gromov} and we do not repeat it here.
\end{proof}

\begin{prop}\label{prop4}
Let $\Omega$, $\psi_n$, $P$ and $\widehat{\Omega}$ be as in Theorem~\ref{thm5}. Let $\sigma_n\colon[a_n,b_n]\to\Omega$ be a sequence of Kobayashi $(A,0)$ quasi geodesic. Define $\tilde{\sigma}_n:=\psi_n\circ\sigma_n$. Suppose that there exists $R>0$ such that
\begin{enumerate}
\item[$\mathrm{(i)}$] $|b_n-a_n|\to\infty$;
\item[$\mathrm{(ii)}$] $\tilde{\sigma}_n([a_n,b_n])\subset \mathbb{B}(0,R)$;
\item[$\mathrm{(iii)}$] $\lim_{n\to\infty}\|\tilde{\sigma}_n(a_n)-\tilde{\sigma}_n(b_n)\|>0$,
\end{enumerate}
then, after a subsequence, there is a $T_n\in[a_n,b_n]$ such that the sequence of $t\mapsto\tilde{\sigma}_n(t+T_n)$ converges uniformly on compact set to an $(A,0)$ quasi geodesic $\tilde{\sigma}\colon\mathbb{R}\to\widehat{\Omega}$.
\end{prop}

\begin{proof}
The proof is similar to \cite[Proposition~5.7]{Fiacchi2020Gromov}, and we give a proof here for the reader's convenience. Suppose
\[
\lim_{n\to\infty}\|\tilde{\sigma}_n(a_n)-\tilde{\sigma}_n(b_n)\|=\epsilon
\]
for some fixed $\epsilon>0$ and choose $T_n\in[a_n,b_n]$ such that
\[
\lim_{n\to\infty}\|\tilde{\sigma}_n(a_n)-\tilde{\sigma}_n(T_n)\|=
\lim_{n\to\infty}\|\tilde{\sigma}_n(T_n)-\tilde{\sigma}_n(b_n)\|=\epsilon/2.
\]
By passing to a subsequence, we may assume that $\tilde{\sigma}_n(T_n)$ converges to a point $y\in\overline{\widehat{\Omega}\cap\mathbb{B}(0,R)}$. If $y\in\widehat{\Omega}$, the Arzel\`a-Ascoli theorem and the stability of the Kobayashi distance (Theorem~\ref{thm1}) indicate that there exists a subsequence such that $\tilde{\sigma}_n$ converges to an $(A,0)$ quasi geodesic $\tilde{\sigma}\colon\mathbb{R}\to\widehat{\Omega}$.

If $y\in\partial\widehat{\Omega}$, in this case we can suppose that
\[
\lim_{n\to\infty}\tilde{\sigma}_n(a_n)\to\xi^-\in\overline{\widehat{\Omega}\cap \mathbb{B}(0,R)}\quad\text{and}\quad
\lim_{n\to\infty}\tilde{\sigma}_n(b_n)\to\xi^+\in\overline{\widehat{\Omega}\cap \mathbb{B}(0,R)}.
\]
If both $\xi^+,\xi^-\in\widehat{\Omega}$, Theorem~\ref{thm1} implies that
\begin{align*}
d_{\widehat{\Omega}}(\xi^+,\xi^-)&=\lim_{n\to\infty} d_{\Omega_n}(\tilde{\sigma}_n(a_n),\tilde{\sigma}_n(b_n))\\
&=\lim_{n\to\infty}d_{\Omega}(\sigma_n(a_n),\sigma_n(b_n))\\
&\geq\lim_{n\to\infty} A^{-1}|b_n-a_n|\\
&=\lim_{n\to\infty}A^{-1}(|b_n-T_n|+|T_n-a_n|)\\
&\geq\lim_{n\to\infty}A^{-2}\big[d_{\Omega}(\sigma(b_n),\sigma_n(T_n))+d_{\Omega}(\sigma_n(T_n),\sigma_n(a_n))
\big]\\
&=\lim_{n\to\infty}A^{-2}\big[d_{\Omega_n}(\tilde{\sigma}_n(b_n),\tilde{\sigma}_n(T_n))+
d_{\Omega_n}(\tilde{\sigma}_n(T_n),\tilde{\sigma}_n(a_n))\big]\\
&=A^{-2}\big[d_{\widehat{\Omega}}(\xi^+,y)+d_{\widehat{\Omega}}(y,\xi^-)\big]=\infty
\end{align*}
since $y\in\partial\widehat{\Omega}$ and this is impossible. Thus one of $\xi^+,\xi^-$ is in $\partial\widehat{\Omega}$. Without loss of generality, we may assume that $\xi^+=:\xi\in\partial\widehat{\Omega}$.

For $n\to\infty$, we have
\[
\max\{|r_n(\tilde{\sigma}_n(t))| : t\in[T_n,b_n]\}\to0.
\]
In fact, if there exists $T_n'\in[T_n,b_n]$ such that $|r_n(\tilde{\sigma}_n(T_n'))|>C>0$ for some $C>0$, the Arzel\`a-Ascoli theorem will drive a contradiction. After a re-parametrization, we may assume that $0\in[T_n,b_n]$ and
\begin{equation}\label{eq5}
|r_n(\tilde{\sigma}_n(0))|=\max\{|r_n(\tilde{\sigma}_n(t))| : t\in[T_n,b_n]\}\to0.
\end{equation}

Corollary~\ref{cor1} implies that the quasi geodesic segments ${\tilde{\sigma}_n}|_{[T_n,b_n]}$ are $L$-Lipschitz, and after a subsequence we assume that these segments converge uniformly on compact subsets to $\hat{\sigma}\colon\mathbb{R}\to\partial\widehat{\Omega}$.

\noindent\textbf{Claim 1:} $\hat{\sigma}$ is constant. In fact, Lemma~\ref{lem6} and Corollary~\ref{cor1} imply that for $t\in[T_n,b_n]$,
\[
\frac{c\|\tilde{\sigma}_n'(t)\|}{|r_n(\tilde{\sigma}_n(t))|^{1/(2m)}}\leq M_{r_n}(\tilde{\sigma}_n(t);\tilde{\sigma}_n'(t))\leq A
\]
and thus
\[
\|\tilde{\sigma}_n'(t)\|\leq Ac^{-1}|r_n(\tilde{\sigma}_n(t))|^{\frac{1}{2m}}.
\]
Consequently for any real numbers $u<v$, we have
\begin{align*}
\|\hat{\sigma}(u)-\hat{\sigma}(v)\|&=\int^v_u\|\hat{\sigma}'(t)\|dt=\lim_{n\to\infty}\int^v_u\|\tilde{\sigma}_n'(t)
\|dt\\
&\leq\lim_{n\to\infty} Ac^{-1}\int^v_u|r_n(\tilde{\sigma}_n(t))|^{\frac{1}{2m}}dt=0
\end{align*}
and thus $\hat{\sigma}$ is constant.

\noindent\textbf{Claim 2:} $\hat{\sigma}$ is not constant. By Lemma~\ref{lem6}, for any $o\in\widehat{\Omega}$,
\begin{align*}
A^{-1}|t|\leq d_{\Omega_n}(\tilde{\sigma}_n(0),\tilde{\sigma}_n(t))&\leq d_{\Omega_n}(\tilde{\sigma}_n(0),o)+d_{\Omega_n}(o,\tilde{\sigma}_n(t))\\
&\leq 2C+A\ln\bigg(\frac{1}{|r_n(\tilde{\sigma}_n(0))r_n(\tilde{\sigma}_n(t))|}\bigg)
\end{align*}
and it follows by \eqref{eq5} that
\[
|r_n(\tilde{\sigma}_n(t))|\leq\sqrt{|r_n(\tilde{\sigma}_n(0))r_n(\tilde{\sigma}_n(t))|}\leq e^{\frac{C}{A}-\frac{1}{2A^2}|t|}.
\]

Now for $T'<b'$, we have
\begin{align*}
&\|\hat{\sigma}(T')-\hat{\sigma}(b')\|\\
=&\lim_{n\to\infty}\|\tilde{\sigma}_n(T')-\tilde{\sigma}_n(b')\|\\
\geq&\lim_{n\to\infty}\big(\|\tilde{\sigma}_n(b_n)-\tilde{\sigma}_n(T_n)\|
-\|\tilde{\sigma}_n(b_n)-\tilde{\sigma}_n(b')\|-\|\tilde{\sigma}_n(T')-\tilde{\sigma}_n(T_n)\|\big)\\
\geq&\|y-\xi\|-\limsup_{n\to\infty}\int^{b_n}_{b'}\|\tilde{\sigma}_n'(t)\|dt
-\limsup_{n\to\infty}\int^{T'}_{T_n}\|\tilde{\sigma}_n'(t)\|dt\\
\geq&\|y-\xi\|-Ac^{-1}\limsup_{n\to\infty}\int^{b_n}_{b'}|r_n(\tilde{\sigma}_n(t))|^{\frac{1}{2m}}dt-Ac^{-1}\limsup_{n\to\infty}
\int^{T'}_{T_n}|r_n(\tilde{\sigma}_n(t))|^{\frac{1}{2m}}dt\\
\geq&\|y-\xi\|-Ac^{-1}\limsup_{n\to\infty}\int^{b_n}_{b'}e^{\frac{1}{2m}(\frac{C}{A}-\frac{1}{2A^2}|t|)}dt-
Ac^{-1}\limsup_{n\to\infty}\int^{T'}_{T_n}e^{\frac{1}{2m}(\frac{C}{A}-\frac{1}{2A^2}|t|)}dt\\
\geq&\|y-\xi\|-Ac^{-1}\limsup_{n\to\infty}\int^{\infty}_{b'}e^{\frac{1}{2m}(\frac{C}{A}-\frac{1}{2A^2}|t|)}dt-
Ac^{-1}\limsup_{n\to\infty}\int^{T'}_{-\infty}e^{\frac{1}{2m}(\frac{C}{A}-\frac{1}{2A^2}|t|)}dt
\end{align*}
and if $-T'$, $b'\to\infty$, then $\hat{\sigma}$ is not constant since $\|y-\xi\|\geq\epsilon/2$.

Now Claim~$1$ and Claim~$2$ are in contradiction and we must have $y\in\widehat{\Omega}$ and the proof is complete.
\end{proof}

\begin{remark}\label{remark3}
It should be noted here that we only use the stability of the Kobayashi distance (Theorem~\ref{thm1}), the uniform estimates of the Kobayashi distances and metrics (Lemma~\ref{lem6}), and Lemma~\ref{lem5} to get the proof of the above Proposition. If we replace $\Omega$ by $\widehat{\Omega}$ and let $\psi_n=\mathrm{id}_{\mathbb{C}^{d+1}}$ be the identity, the ingredients are also satisfied, and thus the conclusion also holds with respect to $\widehat{\Omega}$, $\psi_n=\mathrm{id}_{\mathbb{C}^{d+1}}$.
\end{remark}

\begin{cor}\label{cor4}
Let $\Omega$, $\psi_n$, $P$ and $\widehat{\Omega}$ be as in Theorem~\ref{thm5}. Let $\sigma_n\colon[a_n,b_n]\to\Omega$ be a sequence of $(A,0)$ quasi geodesics, and set $\tilde{\sigma}_n:=\psi_n\circ\sigma_n$. Suppose that $\sigma_n$ converges locally uniformly to an $(A,0)$ quasi geodesic $\tilde{\sigma}\colon\mathbb{R}\to\widehat{\Omega}$. If $\lim_{n\to\infty}\tilde{\sigma}_n(b_n)=:x_\infty$, then
\[
\lim_{t\to+\infty}\tilde{\sigma}(t)=x_\infty.
\]
\end{cor}

\begin{proof}
Assume for a contradiction that $\lim_{t\to+\infty}\tilde{\sigma}(t)\neq x_\infty$, then there exist $t_n\nearrow+\infty$ and some $y_\infty\in\overline{\mathbb{C}^{d+1}}$ such that
\[
y_\infty=\lim_{n\to\infty}\tilde{\sigma}(t_n)\neq x_\infty.
\]
As $\tilde{\sigma}_n$ converges locally uniformly to $\tilde{\sigma}$, there exists a subsequence $t_{n_k}\in[a_{n_k},b_{n_k}]$ with the property
\[
\lim_{k\to\infty}\tilde{\sigma}_{n_k}(t_{n_k})=y_\infty.
\]
As $x_\infty\neq y_\infty$, one of them is finite. Hence there exist a subinterval $[t_{n_k}',b_{n_k}']\subset[t_{n_k},b_{n_k}]$ and $R>0$, $\epsilon>0$ such that the hypotheses of Proposition~\ref{prop4} are satisfied on the interval $[t_{n_k}',b_{n_k}']$. Thus there exits $T_k\in[t_{n_k}',b_{n_k}']$ such that $t\mapsto\tilde{\sigma}_{n_k}(t+T_k)$ converges to an $(A,0)$ quasi geodesic $\hat{\sigma}\colon\mathbb{R}\to\widehat{\Omega}$. Consequently
\[
d_{\widehat{\Omega}}(\tilde{\sigma}(0),\hat{\sigma}(0))=\lim_{k\to\infty}d_{\Omega_{n_k}}(\tilde{\sigma}_{n_k}(0),
\hat{\sigma}_{n_k}(T_k))\geq A^{-1}T_k\geq\lim_{k\to\infty}A^{-1}t_{n_k}'=\infty
\]
which is a contradiction.
\end{proof}

\section{The Gromov product with respect to the Catlin metric}\label{section8}

For every $x,y,o\in\widehat{\Omega}$, the \textit{Gromov product} with respect to the Kobayashi metric is defined by
\[
(x|y)_o:=\frac{1}{2}(d_{\widehat{\Omega}}(x,o)+d_{\widehat{\Omega}}(y,o)-d_{\widehat{\Omega}}(x,y)).
\]
The Gromov product is the main device to study the asymptotic behavior of the geodesics (quasi-geodesics). Due to the estimate of Lemma~\ref{lem4}, following the strategy of \cite{Fiacchi2020Gromov}, we shall consider $M_{r_P}$ instead of the Kobayashi metric $F_{\widehat{\Omega}}$ on $\widehat{\Omega}$. First we give some definitions (cf. \cite{abate1994bookfinsler}):

\begin{defn}
Let $D\subset\mathbb{C}^{d+1}$ be a domain. A \textit{Finsler metric} is an upper semi-continuous function $F\colon D\times\mathbb{C}^{d+1}\to[0,\infty)$ with the following properties
\begin{enumerate}
\item[$\mathrm{(1)}$] $F(p;X)>0$ for all $p\in D$, $X\in\mathbb{C}^{d+1}$, $X\neq0$;
\item[$\mathrm{(2)}$] $F(p;\lambda X)=|\lambda|F(p;X)$ for all $p\in D$, $X\in\mathbb{C}^{d+1}$, and $\lambda\in\mathbb{C}$.
\end{enumerate}
\end{defn}

\begin{defn}
For a Finsler metric $F$ on $D$, denote by
\[
d_{D}^F(p,q)=\inf_{\gamma}\bigg\{\int^1_0F(\gamma(t);\gamma'(t))dt\bigg\}
\]
the corresponding distance, where the infimum taking over all piecewise $\mathcal{C}^1$ curve $\gamma$ such that $\gamma\colon[0,1]\to D$, $\gamma(0)=p$, $\gamma(1)=q$.
\end{defn}

\begin{remark}
The Kobayashi metric $F_{\widehat{\Omega}}(z;X)$ is an important Finsler metric on $\widehat{\Omega}$. Moreover, by definition, the quantity $M_{r_P}(z;X)$ defined in Lemma~\ref{lem4} is also a Finsler metric which is defined globally on $\widehat{\Omega}$, and follow the notation in \cite{Fiacchi2020Gromov}, we call it the \textit{Catlin metric} on $\widehat{\Omega}$. Moreover, we call the corresponding distance $d_{\widehat{\Omega}}^C$ the \textit{Catlin distance} on $\widehat{\Omega}$.
\end{remark}

By Lemma~\ref{lem4}, the Calin metric and the Kobayashi metric on $\widehat{\Omega}$ are equivalent up to a uniform constant $A\geq1$, that is,
\[
A^{-1}M_{r_P}(z;X)\leq F_{\widehat{\Omega}}(z;X)\leq AM_{r_P}(z;X),\quad z\in\widehat{\Omega},\quad X\in\mathbb{C}^{d+1}.
\]
Thus, if $\sigma\colon[a,b]\to\widehat{\Omega}$ is a geodesic with respect to the Catlin metric, it is also an $(A,0)$ quasi-geodesic with respect to the Kobayashi metric and vice versa. It follows that the properties of the Kobayashi $(A,0)$ quasi-geodesics appearing the previous section (more precisely, Corollary~\ref{cor1}, Proposition~\ref{prop4} and Corollary~\ref{cor4}) are also satisfied by the Catlin geodesics.

\begin{remark}
It should be noted here that, we can also define the Catlin metric $M_r$ near the boundary point $u_\infty\in\partial\Omega$ appearing in Theorem~\ref{thm5} locally in a small neighborhood $\Omega\cap W$. But we can not obtain Proposition~\ref{prop4} by using Catlin metric directly since the stability of the Catlin distances $d_{\Omega_n}^C$ under the scaling sequence $\psi_n$ is not clear.
\end{remark}

One advantage to adopt the Catlin metric on $\widehat{\Omega}$ is the following:

\begin{lem}\label{lem12}
For each $x=(x_0,x_1,\ldots,x_n)\in\partial\widehat{\Omega}$ and $a>0$, the curve
\[
\sigma(t):=x-(ae^{-t},0,\ldots,0),\quad t\in\mathbb{R}
\]
is a geodesic respect to the Catlin metric $M_{r_P}$.
\end{lem}

\begin{proof}
The proof is the same as the proof of Lemma~\ref{lem5} (where the constant $A$ do not occur) and we omit it.
\end{proof}

Now for every $x,y,o\in\widehat{\Omega}$, the \textit{Gromov product} with respect to the Caltin metric $M_{r_P}$ is defined by
\[
(x|y)_o^C:=\frac{1}{2}(d_{\widehat{\Omega}}^C(x,o)+d_{\widehat{\Omega}}^C(y,o)-d_{\widehat{\Omega}}^C(x,y)).
\]
For any $o\in\widehat{\Omega}$, we have the following:

\begin{prop}\label{prop2}
Let $p_n,q_n\in\widehat{\Omega}$ be two sequences with $p_n\to\xi^+\in\partial\widehat{\Omega}\cup\{\infty\}$ and $q_n\to\xi^-\in\partial\widehat{\Omega}$, and
\[
\liminf_{m,n\to\infty}(p_n|q_m)_o^C<\infty,
\]
then $\xi^+\neq\xi^-$.
\end{prop}

\begin{proof}
After a subsequence, we may assume that $\lim_{n\to\infty}(p_n|q_n)_o^C<\infty$ exists. Assume for a contradiction that $\xi^+=\xi^-=:\xi\in\partial\widehat{\Omega}$. Set $a_n:=-r_P(p_n)>0$ and $b_n:=-r_P(q_n)>0$ and let
\[
\sigma^+_n(t):=p_n+(a_n-e^{-t},'0),\qquad\sigma^-_n(t):=q_n+(b_n-e^{-t},'0).
\]
Then Lemma~\ref{lem12} confirms that they are Catlin geodesics.

As $r_P(\sigma_n^+(0))=r_P(\sigma_n^-(0))=-1$ and $p_n,q_n\to\xi$, there exists $R>0$ such that for all $n$ large
\[
d_{\widehat{\Omega}}^C(\sigma^+_n(0),o)<R,\quad\text{and}\quad d_{\widehat{\Omega}}^C(\sigma^-_n(0),o)<R.
\]
Thus for fixed $T>0$,
\begin{align*}
d_{\widehat{\Omega}}^C(o,p_n)&\geq d_{\widehat{\Omega}}^C(\sigma^+_n(0),p_n)-d_{\widehat{\Omega}}^C(\sigma^+_n(0),o)\\
&\geq d_{\widehat{\Omega}}^C(\sigma^+_n(0),p_n)-R\\
&= d_{\widehat{\Omega}}^C(\sigma^+_n(0),\sigma^+_n(T))+d_{\widehat{\Omega}}^C(\sigma^+_n(T),p_n)-R\\
&= d_{\widehat{\Omega}}^C(\sigma^+_n(T),p_n)+T-R
\end{align*}
and similarly
\[
d_{\widehat{\Omega}}^C(o,q_n)\geq d_{\widehat{\Omega}}^C(\sigma^-_n(T),q_n)+T-R.
\]

Also we have
\begin{equation*}
d_{\widehat{\Omega}}^C(p_n,q_n)\leq d_{\widehat{\Omega}}^C(p_n,\sigma^+_n(T))+
d_{\widehat{\Omega}}^C(\sigma_n^+(T),\sigma^-_n(T))+d_{\widehat{\Omega}}^C(\sigma^-_n(T),q_n)
\end{equation*}
and consequently the Gromov product has the estimate
\begin{align*}
2(p_n|q_n)_o^C&=d_{\widehat{\Omega}}^C(p_n,o)+d_{\widehat{\Omega}}^C(q_n,o)-d_{\widehat{\Omega}}^C(p_n,q_n)\\
&\geq2T-2R-d_{\widehat{\Omega}}^C(\sigma^+_n(T),\sigma^-_n(T)).
\end{align*}
But
\begin{align*}
\lim_{n\to\infty}d_{\widehat{\Omega}}^C(\sigma^+_n(T),\sigma^-_n(T))&=\lim_{n\to\infty}d_{\widehat{\Omega}}^C(p_n+
(a_n-e^{-T},0'),q_n+(b_n-e^{-T},0'))\\
&=d_{\widehat{\Omega}}^C(\xi+(-e^{-T},0'),\xi+(-e^{-T},0'))=0
\end{align*}
and this implies that
\[
\lim_{n\to\infty}(p_n|q_n)_o^C\geq T-R.
\]
This is impossible since $T$ is arbitrary and the proof is complete.
\end{proof}

\begin{cor}\label{cor2}
Let $p_n,q_n\in\widehat{\Omega}$ be two sequences with $p_n\to\xi^+\in\partial\widehat{\Omega}\cup\{\infty\}$ and $q_n\to\xi^-\in\partial\widehat{\Omega}\cup\{\infty\}$, and
\[
\lim_{n,m\to\infty}(p_n|q_m)_o^C=\infty,
\]
then $\xi^+=\xi^-$.
\end{cor}

\begin{proof}
The proof is similar to \cite[Propostion~11.5]{zimmer2016convexmathann}. Assume for a contradiction that $\xi^+\neq\xi^-$, then at least one of them is finite. After a subsequence we may assume that
\[
\lim_{n\to\infty}(p_n|q_n)_o^C=\infty.
\]
Now let $\sigma_n\colon[a_n,b_n]\to\widehat{\Omega}$ be a sequence of Catlin geodesics (thus Kobayashi  $(A,0)$ quasi geodesics) such that $\sigma_n(a_n)=p_n$, $\sigma_n(b_n)=q_n$. By Proposition~\ref{prop4} and Remark~\ref{remark3} there exists $T_n\in[a_n,b_n]$ such that $\sigma_n(t+T_n)$ converges locally uniformly to a geodesic $\sigma\colon\mathbb{R}\to\widehat{\Omega}$. Since $\sigma_n$ are Catlin geodesics, we have
\begin{align*}
(p_n|q_n)_o^C&=\frac{1}{2}(d_{\widehat{\Omega}}^C(p_n,o)+d_{\widehat{\Omega}}^C(o,q_n)-d_{\widehat{\Omega}}^C(p_n,q_n))\\
&=\frac{1}{2}(d_{\widehat{\Omega}}^C(p_n,o)+d_{\widehat{\Omega}}^C(o,q_n)-d_{\widehat{\Omega}}^C(p_n,\sigma_n(T_n))-
d_{\widehat{\Omega}}^C(\sigma_n(T_n),q_n))\\
&\leq d_{\widehat{\Omega}}^C(o,\sigma_n(T_n)).
\end{align*}
Consequently
\[
\infty=\lim_{n\to\infty}(p_n|q_n)_o^C\leq\lim_{n\to\infty}d_{\widehat{\Omega}}^C(o,\sigma_n(T_n))=
d_{\widehat{\Omega}}^C(o,\sigma(0))
\]
which is a contradiction.
\end{proof}

\begin{lem}\label{lem10}
Let $\sigma\colon\mathbb{R}\to\widehat{\Omega}$ be a Catlin geodesic. Then both
\[
\lim_{t\to-\infty}\sigma(t)\quad\text{and}\quad\lim_{t\to+\infty}\sigma(t)
\]
exist in $\overline{\mathbb{C}^{d+1}}$. Moreover, if one of the limits is finite, then
\[
\lim_{t\to-\infty}\sigma(t)\neq\lim_{t\to\infty}\sigma(t).
\]
\end{lem}

\begin{proof}
First we show $\lim_{t\to+\infty}\sigma(t)$ exists (the limit to $-\infty$ is analogous). Assume for a contradiction that there exists $t_k,s_k\to\infty$ such that
\[
\lim_{k\to\infty}\sigma(t_k)\neq\lim_{k\to\infty}\sigma(s_k),
\]
then the Gromov product
\begin{align*}
\lim(\sigma(t_k)|\sigma(s_k))_{\sigma(0)}^C&=\lim_{k\to\infty}\frac{1}{2}\big(t_k+s_k-|t_k-s_k|\big)\\
&=\lim_{k\to\infty}\min\{t_k,s_k\}\to\infty.
\end{align*}
Now Corollary~\ref{cor2} indicates that $\lim_{k\to\infty}\sigma(t_k)=\lim_{k\to\infty}(s_k)$, which is a contradiction.

If one of the limits is finite, the Gromov product
\[
\big(\sigma(-t)|\sigma(t)\big)_{\sigma(0)}^C\equiv0<\infty
\]
and Proposition~\ref{prop2} ensures that the two limits are distinct.
\end{proof}

\section{Scaling at infinity}\label{section7}

A geodesic line $\sigma\colon\mathbb{R}\to\widehat{\Omega}$ is \textit{well behaved} if both limits
\[
\lim_{t\to-\infty}\sigma(t)\quad\text{and}\quad\lim_{t\to+\infty}\sigma(t)
\]
exist in $\overline{\mathbb{C}^{d+1}}$ and are distinct. By Lemma~\ref{lem10} we know that the limits are different if one of them is finite. Thus we need to study the case when both limits are $\infty$. The strategy is inspired by \cite[Section~12]{zimmer2016convexmathann}, we shall take a dilation on $\widehat{\Omega}$ to understand the asymptotic behavior of the geodesics at infinity. The techniques and methods used in this section are similar to those which are used in the previous sections, and for accuracy and completeness, we give the details here.

\subsection{Stability of the Kobayashi metrics}
We begin with a dilation. Set
\[
\chi_n=\begin{pmatrix}
n^{-1}& & & &\\
&n^{-1/(2m')}& & & \\
& &n^{-1/2}& &\\
& & & \ddots&\\
& & & &n^{-1/2}
\end{pmatrix}
\]
where $2m'$ stands for the degree of $P(z_1,\overline{z}_1)$. Denote by
\[
\rho_n:=\frac{1}{n}r_P\circ\chi_n^{-1}=\re{z_0}+P^{(n)}(z_1)+\sum^d_{\alpha=2}|z_\alpha|^2
\]
the defining function of $\chi_n\widehat{\Omega}$, where $P^{(n)}(z_1)=P(n^{1/(2m')}z_1,n^{1/(2m')}\overline{z}_1)/n$. If write
\[
P(z_1,\overline{z}_1)=P_{2m'}(z_1,\overline{z}_1)+\cdots+P_{2}(z_1,\overline{z}_1)
\]
where $P_{k}$ is a real homogeneous polynomial of degree $2\leq k\leq 2m'$, then $\rho_n$ converges uniformly on compact subsets to
\[
\rho(z):=\re(z_0)+P_{2m'}(z_1,\overline{z}_1)+\sum^d_{\alpha=2}|z_\alpha|^2
\]
and thus $\chi_n\widehat{\Omega}$ converges locally uniformly to the domain
\[
\widetilde{\Omega}:=\{z\in\mathbb{C}^{d+1} : \rho(z)<0\}
\]
as $n\to\infty$.

\begin{remark}
The dilation $\chi_n$ can be seen as a scaling process at infinity in the following sense: Let $u_n=(n,'0)\to\infty$ and let $\Phi_n(z)=z$ be the identity. Then $\chi_n=\chi_n\circ\Phi_n$ and $\chi_n(u_n)=(-1,'0)$. But we can not exploit Theorem~\ref{thm4} and Theorem~\ref{thm1} directly, since the two theorems are based on the engulfing property of the distorted polydiscs near the origin and estimates of the Kobayashi distances on bounded Levi corank one domains \cite[Theorem~1.1]{verma2015corankone}, which are not available on the unbounded limit domains.
\end{remark}

Now we give the estimates of the Kobayashi metrics. If we set
\[
A^{\rho_n}_{\ell}(z_1)=\max\bigg\{\bigg|\frac{\partial^{j+k}\rho_n}{\partial z_1^j\partial\overline{z}_1^k}(z_1)\bigg| :j,k>0,j+k=\ell\bigg\},
\]
then for $z\in\chi_n\widehat{\Omega}$ and $X=(x_0,x_1,\ldots,x_d)\in\mathbb{C}^{d+1}$,
\[
M_{\rho_n}(z;X)=\frac{|x_0+2x_1[P^{(n)}]'(z_1)+2\sum^d_{\alpha=2}x_\alpha\overline{z}_\alpha|}{|\rho_n(z)|}+
|x_1|\sum^{2m'}_{\ell=2}\bigg|\frac{A^{\rho_n}_\ell(z_1)}{\rho_n(z)}\bigg|^{\frac{1}{\ell}}+\sum^d_{\alpha=2}\frac{|x_\alpha|}
{\sqrt{|\rho_n(z)|}}
\]
where $[P^{(n)}]'(z_1)={\partial P^{(n)}}(z_1)/{\partial z_1}$, we have the following:
\begin{lem}\label{lem11}
For all $z\in\chi_n\widehat{\Omega}$ and $X\in\mathbb{C}^{d+1}$,
\[
 M_{r_P}(\chi_n^{-1}(z);(d\chi_n^{-1})_{|z}X)=M_{\rho_n}(z;X)
\]
and $M_{\rho_n}\to M_{\rho}$ locally uniformly on $\widehat{\Omega}$.
\end{lem}

\begin{proof}
The proof is the same as the proof of Theorem~\ref{thm2} where we let $\Phi_n$ be the identity and we omit the proof.
\end{proof}

We need the following theorem \cite[Theorem~2.1.]{yujiye1995limitsofkobayshimetric}:

\begin{thm}
Let $D_n,D$ be a family of domains in $\mathbb{C}^{d+1}$ such that $D_n$ converges in the Hausdorff topology to $D$. Let $D$ be a taut domain. If, there exists another taut domain $D'$ such that $D_n\subset D'$ for all large $n$, then we have
\[
\lim_{n\to\infty}F_{D_n}(z;X)=F_D(z;X)\quad (z,X)\in D\times\mathbb{C}^{d+1}.
\]
Moreover, the convergence takes place uniformly over compact subsets of $D\times\mathbb{C}^{d+1}$.
\end{thm}

\begin{prop}\label{prop3}
For $(z;X)\in\widetilde{\Omega}\times\mathbb{C}^{d+1}$,
\[
\lim_{n\to\infty}F_{\chi_n\widehat{\Omega}_n}(z;X)=F_{\widetilde{\Omega}}(z;X).
\]
Moreover, the convergence is uniform on compact subsets of $\widetilde{\Omega}\times\mathbb{C}^{d+1}$.
\end{prop}

\begin{proof}
As $\chi_n\widehat{\Omega}$ is pseudo-convex, the limit domain $\widetilde{\Omega}$ is also pseudo-convex and thus $P_{2m'}(z_1,\overline{z}_1)$ is real homogeneous of degree $2m'$, subharmonic, contains no harmonic terms. Now $\widetilde{\Omega}$ is a WB-domain in the sense of \cite[Definition~4.1]{gaussier2016positivity} and thus it is taut. We only need to find a bumping domain $D'$ which is taut and contains $\chi_n\widehat{\Omega}$ for all $n$ large.

Let $\epsilon_0>0$ such that $P_{2m'}(z_1,\overline{z}_1)-2\epsilon_0|z_1|^{2m'}$ is still subharmonic. Then there exists constant $R_0>0$ such that
\[
P(z_1,\overline{z}_1)\geq P_{2m'}(z_1,\overline{z}_1)-\epsilon_0|z_1|^{2m'}
\]
for all $|z_1|\geq R_0$ and consequently there is a constant $C_0>0$ such that
\[
P(z_1,\overline{z}_1)\geq P_{2m'}(z_1,\overline{z}_1)-\epsilon_0|z_1|^{2m'}-C_0
\]
for all $z_1\in\mathbb{C}$. Similar the remaining terms of $P^{(n)}(z_1)$ where the degree $<2m'$ will be controlled by $\epsilon_0|z_1|^{2m'}+C'$ for some constant $C'>0$ and thus $\chi_n\widehat{\Omega}\subset D'$ where
\[
D':=\{z\in\mathbb{C}^{d+1} : \re(z_0)+P_{2m'}(z_1,\overline{z}_1)-2\epsilon_0|z_1|^{2m'}+\sum^d_{\alpha=2}|z_\alpha|^2-C_0-C'<0\}
\]
for all $n$ large. As $D'$ is also weighted homogeneous (after a shift of constant $C_0+C'$), hence complete hyperbolic and the proof is complete.
\end{proof}

Now we can give the global estimate of the Kobayashi metric on $\widetilde{\Omega}$:

\begin{lem}\label{lem13}
There exits $A\geq1$ such that
\[
A^{-1}M_\rho(z;X)\leq F_{\widetilde{\Omega}}(z;X)\leq AM_\rho(z;X)
\]
for all $(z;X)\in\widetilde{\Omega}\times\mathbb{C}^{d+1}$.
\end{lem}

\begin{proof}
By Theorem~\ref{thm3}, the previous proposition and Lemma~\ref{lem11}, since $M_{\rho_n}\to M_\rho$ the proof is the same as Lemma~\ref{lem4}.
\end{proof}

\subsection{Stability of the Kobayashi distances}

Now we show that the Kobayashi distances are also stable under the dilation $\chi_n$. First we prove the following:

\begin{lem}\label{lem8}
For any $p,q\in\widetilde{\Omega}$,
\begin{equation*}
\limsup_{n\to\infty}d_{\chi_n\widehat{\Omega}}(p,q)\leq d_{\widetilde{\Omega}}(p,q),
\end{equation*}
and the convergence is uniform on compact subsets of $\widetilde{\Omega}\times\widetilde{\Omega}$.
\end{lem}

\begin{proof}
The proof is the same as Lemma~\ref{lem2} by Proposition~\ref{prop3} and the complete hyperbolicity of $\widetilde{\Omega}$.
\end{proof}

For any fixed $R>0$, denote by $\mathbf{B}_{\chi_n\widehat{\Omega}}((-1,'0),R)$ the Kobayashi ball of $\chi_n\widehat{\Omega}$ centered at $(-1,'0)$ with radius $R$, then we have the following:

\begin{lem}\label{lem9}
For any $R>0$ fixed, the ball $\mathbf{B}_{\chi_n\widehat{\Omega}}((-1,'0),R)$ is uniformly compactly contained in $\widetilde{\Omega}$ for all $n$ large.
\end{lem}

\begin{proof}
For any $R>0$ fixed, assume for a contradiction that there exists a sequence $\{p_n\}\subset\mathbf{B}_{\chi_n\widehat{\Omega}}((-1,'0),R)$ which cluster at some point $\xi\in\partial\widetilde{\Omega}\cup\{\infty\}$. Thus $\rho(\xi)=0$ or $\rho(\xi)=-\infty$. Now \eqref{eq7} (replace $r_P$ by $\rho_n$) implies that
\begin{align*}
R\geq d_{\chi_n\widehat{\Omega}}((-1,'0),p_n)\geq A^{-1}\big|\log|\rho_n(p_n)|\big|\to A^{-1}\big|\log|\rho(\xi)|\big|=\infty
\end{align*}
as $n\to\infty$, which is impossible and the proof is complete.
\end{proof}

\begin{prop}\label{prop6}
For any $p,q\in\widetilde{\Omega}$,
\begin{equation*}
\lim_{n\to\infty}d_{\chi_n\widehat{\Omega}}(p,q)= d_{\widetilde{\Omega}}(p,q),
\end{equation*}
and the convergence is uniform on compact subsets of $\widetilde{\Omega}\times\widetilde{\Omega}$.
\end{prop}

\begin{proof}
Let $K\Subset\widetilde{\Omega}$ be the compact subset which contains $(-1,'0)$. As $\widetilde{\Omega}$ is complete hyperbolic,
\[
\widetilde{\Omega}=\bigcup_{\nu=1}\mathbf{B}_{\widetilde{\Omega}}((-1,'0),\nu)
\]
and we choose $\nu_K$ such that $K\subset\mathbf{B}_{\widetilde{\Omega}}((-1,'0),\nu_K)$. Also Lemma~\ref{lem8} implies that there exists $C_K>0$ and $N>0$ such that $\mathbf{B}_{\widetilde{\Omega}}((-1,'0),\nu_K)
\subset\mathbf{B}_{\chi_n\Omega}((-1,'0),C_K)$ for all $n>N$.

For any $p,q\in K$, we have
\[
d_{\chi_n\widehat{\Omega}}(p,q)\leq d_{\chi_n\widehat{\Omega}}(p,(-1,'0))+d_{\chi_n\widehat{\Omega}}((-1,'0),q)\leq 2C_K
\]
for all $n>N$. For small $\epsilon>0$, chose $R'\gg8C_K$ such that
\[
\tanh(R'-2C_K)\geq1-\epsilon.
\]
It follows by Lemma~\ref{lem1} ($\chi_n\widehat{\Omega}$ and $\mathbf{B}_{\chi_n\widehat{\Omega}}((-1,'0),2R')$ play the role of $D$ and $D'$, respectively, and $b=R'$) that
\begin{align*}
d_{\mathbf{B}_{\chi_n\widehat{\Omega}}((-1,'0),2R')}(p,q)&\leq\frac{d_{\chi_n\widehat{\Omega}}(p,q)}
{\tanh(R'-d_{\chi_n\widehat{\Omega}}(p,q))}\\
&\leq\frac{d_{\chi_n\widehat{\Omega}}(p,q)}{\tanh(R'-2C_K)}\leq\frac{1}{1-\epsilon}
d_{\chi_n\widehat{\Omega}}(p,q)
\end{align*}
for all $n>N$ large. Now Lemma~\ref{lem9} implies that
\[
\mathbf{B}_{\chi_n\widehat{\Omega}}((-1,'0),2R')\subset\widetilde{\Omega},
\]
uniformly for all $n$ large and the decreasing property of the Kobayashi distance indicates that
\[
d_{\widetilde{\Omega}}(p,q)\leq d_{\mathbf{B}_{\chi_n\widehat{\Omega}}((-1,'0),2R')}(p,q)\leq
\frac{1}{1-\epsilon}
d_{\chi_n\widehat{\Omega}}(p,q)
\]
and let $n\to\infty$, we have
\[
d_{\widetilde{\Omega}}(p,q)\leq\frac{1}{1-\epsilon} \liminf_{n\to\infty}
d_{\chi_n\widehat{\Omega}}(p,q).
\]
Combine Lemma~\ref{lem8} and the arbitrariness of $\epsilon$ we finish the proof.
\end{proof}

\subsection{Asymptotic behavior of the geodesics}
In order to show that all geodesics in $\widehat{\Omega}$ are well behaved, we need to get uniform estimates of the Kobayashi metrics on the scaled domains $\chi_n\widehat{\Omega}$.

\begin{lem}
For each $R>0$, there exist $c>0$ and $C>0$ such that for all $n$ large,
\[
M_{\rho_n}(z;X)\geq\frac{c\|X\|}{|\rho_n(z)|^{1/{(2m')}}},\quad z\in\chi_n\widehat{\Omega}\cap \mathbb{B}(0,R),\quad X\in\mathbb{C}^{d+1}
\]
and for any $o\in\widetilde{\Omega}$
\[
d_{\chi_n\widehat{\Omega}}(z,o)\leq C+A\ln\bigg(\frac{1}{|\rho_n(z)|}\bigg),\quad z\in \chi_n\widehat{\Omega}\cap \mathbb{B}(0,R).
\]
\end{lem}

\begin{proof}
For each $R>0$, as $\rho_n\to\rho$ uniformly there exists $C_*>0$ such that $|\rho_n(z)|$, $|\rho(z)|\leq C_*$ for all $n$ large and for all $z\in\mathbb{B}(0,R)$. Thus, for every $\ell\in\{2,\ldots,2m'\}$,
\[
|\rho_n(z)|^{\frac{1}{\ell}}\leq C_*^{\frac{1}{\ell}-\frac{1}{2m'}}|\rho_n(z)|^{\frac{1}{2m'}}.
\]
Moreover, there exists $B>0$ such that $\sum^{2m'}_{\ell=2}(A^{\rho}_\ell(z_1))^{1/\ell}\geq 2B$ for all $z\in\mathbb{B}(0,R)$, and as
\[
\sum^{2m'}_{\ell=2}(A^{\rho_n}_\ell(z_1))^{1/\ell}\to\sum^{2m'}_{\ell=2}(A^{\rho}_\ell(z_1))^{1/\ell},
\]
it follows that $\sum^{2m'}_{\ell=2}(A^{\rho_n}_\ell(z_1))^{1/\ell}\geq B$ for all $z\in\mathbb{B}(0,R)$ and for all $n$ large. By the same argument we see that there exists $D>0$ such that
\[
\big|2[P^{(n)}(z_1)]'\big|\leq D, \quad \big|2z_2\big|,\ldots,\big|2z_d\big|\leq D
\]
for all $n$ large and $z=(z_0,z_1,\ldots,z_d)\in\mathbb{B}(0,R)$.

Now consider the inequality ($x_k,y,a_k\in\mathbb{C}$, and $\sum^d_{k=1}|a_k|\leq |a|$)
\begin{align*}
\bigg|y+\sum^d_{k=1}a_kx_k\bigg|+\sum^d_{k=1}|x_k|&\geq\frac{1}{1+\sum^d_{k=1}|a_k|}
\bigg(|y|-\sum^d_{k=1}|a_k||x_k|\bigg)+
\sum^d_{k=1}\frac{1+|a_k|}{1+|a|}|x_k|\\
&\geq\frac{1}{1+|a|}\bigg(\sum^d_{k=1}|x_k|+|y|\bigg)
\end{align*}
we have
\begin{align*}
&M_{\rho_n}(p;X)\\
=&\frac{|x_0+2x_1[P^{(n)}]'(z_1)+2\sum^d_{\alpha=2}x_\alpha\overline{z}_\alpha|}{|\rho_n(z)|}+
|x_1|\sum^{2m'}_{\ell=2}\bigg(\frac{A^{\rho_n}_\ell(z_1)}{|\rho_n(z)|}\bigg)^{1/\ell}+\sum^d_{\alpha=2}\frac{|x_\alpha|}
{\sqrt{|\rho_n(z)|}}\\
\geq&\frac{1}{|\rho_n(p)|^{\frac{1}{2m'}}}\bigg(C_*^{\frac{1}{2m'}-1}
\bigg|x_0+2x_1[P^{(n)}]'(z_1)+2\sum^d_{\alpha=2}x_\alpha\overline{z}_\alpha\bigg|
+C_*^{\frac{1}{2m'}}B|x_1|+C_*^{\frac{1}{2m'}-\frac{1}{2}}\sum^d_{k=2}|x_k|\bigg)\\
\geq&\frac{C^*}{|\rho_n(p)|^{\frac{1}{2m'}}}\frac{1}{1+dD}\big(|x_0|+\sum^d_{k=1}|x_k|\big)
\end{align*}
where $C^*:=\min\{C_*^{\frac{1}{2m'}-1},C_*^{\frac{1}{2m'}-\frac{1}{2}},C_*^{\frac{1}{2m'}}B\}$.

For any fixed $R>0$, by compactness there exists $C_R>0$ such that
\[
d_{\widetilde{\Omega}}(x,o)\leq 2C_R,\quad\text{for all}\quad x\in\mathbb{B}(0,R),\quad \rho(x)=\rho(o)
\]
and as $d_{\chi_n\widehat{\Omega}}$ converges uniformly on compact subsets of $\widetilde{\Omega}\times\widetilde{\Omega}$ to $d_{\widetilde{\Omega}}$, we can suppose that
\[
d_{\chi_n\widehat{\Omega}}(x,o)\leq C_R,\quad\text{for all}\quad x\in\mathbb{B}(0,R),\quad\rho_n(x)=\rho_n(o)
\]
for all $n$ large. Next for each $p\in\chi_n\widehat{\Omega}\cap\mathbb{B}(0,R)$, let $x=p+(\lambda,0')\in\chi_n\widehat{\Omega}$ with $\lambda\in\mathbb{R}$ such that $\rho_n(x)=\rho_n(o)$. A direction computation shows that
\[
d_{\chi_n\widehat{\Omega}}(x,p)\leq A\bigg|\ln\bigg(\frac{\rho_n(x)}{\rho_n(p)}\bigg)\bigg|=A
\bigg|\ln\bigg(\frac{\rho_n(o)}{\rho_n(p)}\bigg)\bigg|.
\]
Now for each $n\in\mathbb{N}$, there exists $A_R>0$ such that for $p\in\chi_n\widehat{\Omega}\cap\mathbb{B}(0,R)$, $o$, the corresponding $x$ are contained in $\mathbb{B}(0,A_R)$. It follows that
\begin{align*}
d_{\chi_n\widehat{\Omega}}(p,o)&\leq d_{\chi_n\widehat{\Omega}}(p,x)+d_{\chi_n\widehat{\Omega}}(x,o)\\
&\leq A\bigg|\ln\bigg(\frac{\rho_n(o)}{\rho_n(p)}\bigg)\bigg|+C_R\\
&\leq A\ln\bigg(\frac{1}{|\rho_n(p)|}\bigg)+A|\ln A_R|+C_R
\end{align*}
and the proof is complete if we denote by $C:=A|\ln A_R|+C_R$.
\end{proof}

\begin{prop}\label{prop5}
Let $\sigma_n\colon[a_n,b_n]\to\widehat{\Omega}$ be a sequence of Kobayashi $(A,0)$ quasi geodesic. Define $\tilde{\sigma}_n:=\chi_n\circ\sigma_n$. Suppose that there exists $R>0$ such that
\begin{enumerate}
\item[$\mathrm{(i)}$] $|b_n-a_n|\to\infty$;
\item[$\mathrm{(ii)}$] $\tilde{\sigma}_n([a_n,b_n])\subset \mathbb{B}(0,R)$;
\item[$\mathrm{(iii)}$] $\lim_{n\to\infty}\|\tilde{\sigma}_n(a_n)-\tilde{\sigma}_n(b_n)\|>0$,
\end{enumerate}
then, after a subsequence, there is a $T_n\in[a_n,b_n]$ such that the sequence of $t\mapsto\tilde{\sigma}_n(t+T_n)$ converges uniformly on compact set to an $(A,0)$ quasi geodesic $\tilde{\sigma}\colon\mathbb{R}\to\widetilde{\Omega}$.
\end{prop}

\begin{proof}
By the previous lemma and the regularity of the quasi-geodesics (Corollary~\ref{cor1}), the proof is the same as Proposition~\ref{prop4} and we do not repeat it here.
\end{proof}

The following lemma is very similar to \cite[Proposition~12.2]{zimmer2016convexmathann}:

\begin{lem}
There exists no Catlin geodesic $\sigma\colon\mathbb{R}\to\widehat{\Omega}$ such that
\begin{equation}\label{eq6}
\lim_{t\to\infty}\sigma(t)=\lim_{t\to-\infty}\sigma(t)=\infty.
\end{equation}
\end{lem}

\begin{proof}
Assume for a contradiction that there is a Catlin geodesic line $\sigma\colon\mathbb{R}\to\widehat{\Omega}$ such that
\[
\lim_{t\to\infty}\sigma(t)=\lim_{t\to-\infty}\sigma(t)=\infty.
\]
Then Lemma~\ref{lem13} implies that $\sigma$ is a Kobayashi $(A,0)$ quasi geodesic. Now consider the $(A,0)$ quasi geodesics $\tilde{\sigma}_n:=\chi_n\circ\sigma\colon\mathbb{R}\to\chi_n\widehat{\Omega}$. Let $R>0$ be a constant such that $R>2\|\sigma(0)\|$. For each $n\in\mathbb{N}$, since
\[
\lim_{t\to\infty}\tilde{\sigma}_n(t)=\lim_{t\to-\infty}\tilde{\sigma}_n(t)=\infty,
\]
and $\|\tilde{\sigma}_n(0)\|<\|\sigma(0)\|$, we can chose $b_n\in(0,\infty)$ such that $\|\sigma_n(b_n)\|=2R$. As $\lim_{n\to\infty}\tilde{\sigma}_n(t_0)=0\in\partial\widetilde{\Omega}$ for fixed $t_0\in\mathbb{R}$ we must have $b_n\to\infty$. Now the hypotheses of Proposition~\ref{prop5} are satisfied (on the interval $[0,b_n]$), and thus there exists $T'_n\in[0,b_n]$ such that $t\mapsto\tilde{\sigma}_n(t+T_n')$ converges locally uniformly to an $(A,0)$ quasi geodesic $\hat{\sigma}_1\colon\mathbb{R}\to\widetilde{\Omega}$. Similarly, by considering $(-\infty,0]$, there exists $T_n''\in[d_n,0]$ such that $t\mapsto\tilde{\sigma}_n(t+T_n'')$ converges to a geodesic $\hat{\sigma}_2\colon\mathbb{R}\to\widetilde{\Omega}$.

Since $\tilde{\sigma}_n(0)\to0$ and $0\in\partial\widetilde{\Omega}$, we must have $T_n'$ and $T_n''$ converges to $\infty$ and $-\infty$, respectively. Thus Proposition~\ref{prop6} indicates that
\begin{align*}
d_{\widetilde{\Omega}}(\hat{\sigma}_1(0),\hat{\sigma}_2(0))&=\lim_{n\to\infty}d_{\chi_n\widehat{\Omega}}
(\tilde{\sigma}_n(T_n'),\tilde{\sigma}_n(T_n''))\\
&=\lim_{n\to\infty}d_{\widehat{\Omega}}(\sigma(T_n'),\sigma(T_n''))\geq\lim_{n\to\infty}A^{-1}|T_n'-T_n''|=\infty
\end{align*}
which is a contradiction and the proof is complete.
\end{proof}

\begin{prop}
The Catlin geodesics in $\widehat{\Omega}$ are well behaved, i.e., if $\sigma\colon\mathbb{R}\to\widehat{\Omega}$ is a Catlin geodesic, then both
\[
\lim_{t\to-\infty}\sigma(t)\quad\text{and}\quad\lim_{t\to\infty}\sigma(t)
\]
exist in $\overline{\mathbb{C}^{d+1}}$ and they are distinct.
\end{prop}

\begin{proof}
By Lemma~\ref{lem10}, we know that the forward and backward limits of a Catlin geodesic exist, and the limits are distinct if at least one is not $\infty$. Now the previous lemmas shows that a Catlin geodesic can not have $\infty$ as a backward and forward limit and the proof is complete.
\end{proof}

By the following corollary, $(A,B)$ quasi geodesics are also well behaved  \cite[Corollary~5.13]{Fiacchi2020Gromov} (cf. \cite[Chapter 5]{Ghys1990Gromov})
\begin{cor}
Let $\sigma\colon\mathbb{R}\to\widehat{\Omega}$ be an $(A,B)$ quasi-geodesic line respect to the Catlin metric. Then both
\[
\lim_{t\to-\infty}\sigma(t)\qquad\lim_{t\to\infty}\sigma(t)
\]
exist in $\mathbb{C}^{d+1}$ and they are different.
\end{cor}

As Kobayashi geodesics are Catlin $(A,0)$ quasi geodesics on $\widehat{\Omega}$, it follows that

\begin{cor}\label{prop7}
The Kobayashi geodesics in $\widehat{\Omega}$ are well behaved, i.e., if $\sigma\colon\mathbb{R}\to\widehat{\Omega}$ is a Kobayashi geodesic, then both
\[
\lim_{t\to-\infty}\sigma(t)\quad\text{and}\quad\lim_{t\to\infty}\sigma(t)
\]
exist in $\overline{\mathbb{C}^{d+1}}$ and they are distinct.
\end{cor}

\section{Proof of Theorem~\ref{thm0}}

In this section we shall give the proof of the main theorem. As bounded pseudo-convex domains of finite type are always \textit{Goldilocks} domains \cite[Definition~1.1]{zimmer2017goldilocks}, we have the following visibility  theorem~{\cite[Theorem~1.4.]{zimmer2017goldilocks}}:

\begin{thm}\label{thm6}
Let $\Omega\subset\mathbb{C}^{d+1}$ be a smoothly bounded pseudo-convex domain of finite type. Suppose the Levi form of all boundary points of $\Omega$ has rank at least $d-1$ and let $A\geq1, B\geq0$. If $\xi,\eta\in\partial\Omega$ and $V_\xi$ and $V_\eta$ are neighborhoods of $\xi$, $\eta$ in $\overline{\Omega}$ so that $\overline{V_\xi}\cap\overline{V_\eta}=\emptyset$ then there exists a compact set $K\subset\Omega$ with the following property: if $\sigma\colon[0,T]\to\Omega$ is an $(A,B)$ quasi-geodesic with $\sigma(0)\in V_\xi$ and $\sigma(T)\in V_\eta$ then $\sigma\cap K\neq\emptyset$.
\end{thm}

\begin{cor}\label{cor3}
Let $\Omega\subset\mathbb{C}^{d+1}$ be a smoothly bounded pseudo-convex domain of finite type and the Levi form of every boundary points of $\Omega$ has rank at least $d-1$. Let $p_n,q_n\in\Omega$ be two sequences with $p_n\to\xi^+\in\partial\Omega$ and $q_n\to\xi^-\partial\Omega$, and
\[
\lim_{n,m\to\infty}(p_n|q_m)_o=\infty,
\]
then $\xi^+=\xi^-$ (here the Gromov product is with respect to the Kobayashi distance).
\end{cor}

\begin{proof}
The proof is the same as \cite[Corollary~6.2.]{Fiacchi2020Gromov}. Assume for a contradiction that $\xi^+\neq\xi^-$ and consider $\sigma_n\colon[a_n,b_n]\to\Omega$ be a sequence of Kobayashi geodesics such that $\sigma_n(a_n)=p_n$ and $\sigma_n(b_n)=q_n$. By the visibility theorem, there exists a compact subset $K\subset\Omega$ and $T_n\in[a_n,b_n]$ such that $\sigma_n(T_n)\in K$. Thus
\[
(p_n|q_n)_o=(\sigma_n(a_n)|\sigma_n(b_n))_o\leq d_{\Omega}(\sigma_n(T_n),o)\leq\max\{d_{\Omega}(x,o) : x\in K\}<\infty
\]
which is a contradiction.
\end{proof}

\begin{prop}\label{prop1}
Let $\Omega\subset\mathbb{C}^{d+1}$ be a smoothly bounded pseudo-convex domain of finite type and the Levi form of every boundary points of $\Omega$ has rank at least $d-1$. Let $\sigma\colon\mathbb{R}\to\Omega$ be a Kobayashi geodesic line. Then both
\[
\lim_{t\to-\infty}\sigma(t),\qquad\lim_{t\to\infty}\sigma(t)
\]
exist in $\partial\Omega$ and they are distinct.
\end{prop}

\begin{proof}
By Corollary~\ref{cor3} the existence of the limits is the same as Lemma~\ref{lem10} (here we use the Gromov product with respect to the Kobayashi distance). Now we show they are different. Assume for a contradiction that there exists a Kobayashi geodesic $\sigma\colon\mathbb{R}\to\Omega$ such that $\lim_{t\to\infty}\sigma(t)=\lim_{t\to-\infty}\sigma(t)=\xi\in\partial\Omega$. We may suppose that $\xi$ is the origin. Let $R>0$ be a constant such that the Kobayashi metric $F_\Omega$ is comparable with $M_r$ on $\Omega\cap\mathbb{B}(0,R)$ (Theorem~\ref{thm3}). Now set
\begin{align*}
\tau^+&:=\inf\{t\in[0,\infty) : \sigma([t,\infty))\subset \mathbb{B}(0,R/2)\}\\
\tau^-&:=\sup\{t\in(-\infty,0] : \sigma((-\infty,t])\subset\mathbb{B}(0,R/2)\}
\end{align*}
and denote by
\[
\sigma^+:=\sigma|_{[\tau^+,\infty)}\quad\text{and}\quad\sigma^-:=\sigma|_{(-\infty,\tau^-]}.
\]

Now let $\{\psi_n\}$, $P\colon\mathbb{C}\to\mathbb{R}$ be as in Theorem~\ref{thm5} (where we set $u_n=(-1/n,'0)$ and $u_\infty=\xi$) and denote by
\[
\tilde{\sigma}_n^+:=\psi_n\circ\sigma^+\quad\text{and}\quad\tilde{\sigma}_n^-:=\psi_n\circ\sigma^-.
\]
Notice that
\[
\lim_{t\to\infty}\tilde{\sigma}_n^+(t)=\lim_{t\to-\infty}\tilde{\sigma}_n^-(t)=0
\]
for each $n\in\mathbb{C}$ and
\[
\lim_{n\to\infty}\tilde{\sigma}_n^+(\tau^+)=\lim_{n\to\infty}\tilde{\sigma}_n^-(\tau^-)=\infty.
\]
Thus the hypotheses of Proposition~\ref{prop4} are satisfied and consequently, there exists $T_n^+>0$ and $T_n^-<0$ such that $t\mapsto\tilde{\sigma}_n^+(t+T_n^+)$ and $t\mapsto\tilde{\sigma}_n^-(t+T_n^-)$ converges locally uniformly to geodesics $\hat{\sigma}^+\colon\mathbb{R}\to\widehat{\Omega}$ and $\hat{\sigma}^-\colon\mathbb{R}\to\widehat{\Omega}$, respectively.

Now for fixed $t_0\geq \tau^+$, $\lim_{n\to\infty}\tilde{\sigma}_n^+(t_0)\to0\in\partial\widehat{\Omega}$, we must have $T_n^+\to\infty$ since $\tilde{\sigma}_n(T_n^+)\to\hat{\sigma}^+(0)\in\widehat{\Omega}$ and similarly $T_n^-\to-\infty$. It follows by Theorem~\ref{thm1} that
\begin{align*}
d_{\widehat{\Omega}}(\hat{\sigma}^+(0),\hat{\sigma}^-(0))&=\lim_{n\to\infty}d_{\Omega_n}(\tilde{\sigma}^+_n(T_n^+),
\tilde{\sigma}_n^-(T_n^-))\\
&=\lim_{n\to\infty}d_{\Omega}(\sigma(T_n^+),\sigma(T_n^-))\\
&=\lim_{n\to\infty}|T_n^+-T_n^-|=\infty
\end{align*}
which is a contradiction and the proof is complete.
\end{proof}

Now we are ready to give the proof of the main theorem. The procedure is similar to the proof of \cite[Theorem~1.1]{Fiacchi2020Gromov} and we present it for completeness.

\begin{proof}[Proof of Theorem~\ref{thm0}]
Assume for a contradiction that $(\Omega,d_\Omega)$ is not Gromov hyperbolic with respect to the Kobayashi metric. Then there exist three sequences of points $\{x_n\}_{n\in\mathbb{N}}$, $\{y_n\}_{n\in\mathbb{N}}$, $\{z_n\}_{n\in\mathbb{N}}$ and Kobayashi geodesic segments $\sigma_{x_n,y_n}$, $\sigma_{y_n,z_n}$, $\sigma_{x_n,z_n}$ connecting them, and a point $u_n\in\sigma_{x_n,y_n}$ such that
\[
d_{\Omega}(u_n,\sigma_{x_n,z_n}\cup\sigma_{y_n,z_n})\geq n.
\]
Passing to a subsequence, we may assume that $u_n\to u_\infty\in\overline{\Omega}$.

\noindent\textbf{Case 1:} $u_\infty\in\Omega$. We can assume that $x_n$, $y_n$, $z_n$ converges to $x_\infty$, $y_\infty$, $z_\infty\in\overline{\Omega}$, respectively. Since
\[
d_{\Omega}(u_n,\{x_n,z_n,y_n\})\geq d_{\Omega}(u_n,\sigma_{x_n,z_n}\cup\sigma_{y_n,z_n})\geq n,
\]
we must have $x_\infty,y_\infty,z_\infty\in\partial\Omega$. After a re-parametrization we can assume that $\sigma_{x_n,y_n}(0)=u_n$. By the Arzel\`a-Ascoli theorem, $\sigma_{x_n,y_n}$ converges uniformly on compact subsets to a geodesic $\sigma\colon\mathbb{R}\to{\Omega}$. The previous proposition confirms that $x_\infty\neq y_\infty$, and thus one of them is different from $z_\infty$. Without loss of generality we assume that $x_\infty\neq z_\infty$. Now by the visibility Theorem~\ref{thm6} there exist a compact $K\subset\Omega$ and $T_n\in\mathbb{R}$ with the property $\sigma_{x_n,z_n}(T_n)\in K$. Thus
\begin{align*}
\max\{d_{\Omega}(u_\infty,x) : x\in K\}&\geq\lim_{n\to\infty}d_{\Omega}(\sigma(0),\sigma_{x_n,z_n}(T_n))\\
&=\lim_{n\to\infty}d_{\Omega}(\sigma_{x_n,y_n}(0),\sigma_{x_n,z_n}(T_n))\\
&\geq\lim_{n\to\infty}d_{\Omega}(u_n,\sigma_{x_n,z_n}\cup\sigma_{y_n,z_n})=\infty
\end{align*}
which is a contradiction.

\noindent\textbf{Case 2:} $u_\infty\in\partial\Omega$. In this case let $R>0$ be number such that the Kobayashi metric $F_\Omega$ is comparable with $M_r$ for $z\in\Omega\cap\mathbb{B}(0,R)$ in Theorem~\ref{thm3}. After a re-parametrization we assume that $\sigma_{x_n,y_n}(0)=u_n$ and let $\psi_n$, $P\colon\mathbb{C}\to\mathbb{R}$ be as in Theorem~\ref{thm5}, and after a subsequence, suppose that $\psi_n(x_n)$, $\psi_n(y_n)$, $\psi_n(z_n)$ converge to $\hat{x}_\infty$, $\hat{y}_\infty$, $\hat{z}_\infty$ in $\overline{\widehat{\Omega}}\cup\{\infty\}$, respectively. In this case we have $\psi_n\circ\sigma_{x_n,y_n}(0)=v_n\to(-1,'0)\in\widehat{\Omega}$.

\textbf{Subcase 1:} If both $\hat{x}_\infty$ and $\hat{y}_\infty$ are in $\partial\widehat{\Omega}$, then $\sigma_{x_n,y_n}\subset \mathbb{B}(0,R)$ for all $n$ large enough. It follows by Arzel\`a-Ascoli theorem and Theorem~\ref{thm1} that $\psi_n\circ\sigma_{x_n,y_n}$ converges to a geodesic $\hat{\sigma}\colon\mathbb{R}\to\widehat{\Omega}$. Now Proposition~\ref{prop7} implies that geodesics in $\widehat{\Omega}$ are well behaved and thus $\hat{x}_\infty\neq\hat{y}_\infty$, thus one them is different from $\hat{z}_\infty$, and we assume that $\hat{x}_\infty\neq\hat{z}_\infty$.

Now if $\sigma_{x_n,z_n}$ is defined on the interval $[a_n,b_n]$ such that $\sigma_{x_n,z_n}(a_n)=x_n$ and $\sigma_{x_n,z_n}(b_n)=z_n$, define
\[
t_n=\sup\{t\in[a_n,b_n] :\sigma_{x_n,z_n}([a_n,t_n])\subset\mathbb{B}(0,R/2)\}
\]
and set $z_n':=\sigma_{x_n,z_n}(t_n)$. Consequently we have $\lim_{n\to\infty}\psi_n(z_n')=\lim_{n\to\infty}\psi_n(z_n)=\hat{z}_\infty$ and the geodesic segment $\sigma_{x_n,z_n'}$ is contained in $\mathbb{B}(0,R)$. Now Proposition~\ref{prop4} indicates that $\psi\circ\sigma_{x_n,z_n'}$ converges locally uniformly to the geodesic $\hat{\sigma}\colon\mathbb{R}\to\widehat{\Omega}$. Thus it follows by Theorem~\ref{thm1} that
\begin{align*}
d_{\widehat{\Omega}}((-1,'0),\hat{\sigma}(0))&=\lim_{n\to\infty}d_{\Omega_n}(\psi_n(u_n),\psi_n\circ\sigma_{x_n,z_n'}
(0))\\
&=\lim_{n\to\infty}d_{\Omega}(u_n,\sigma_{x_n,z_n'}(0))\\
&\geq\lim_{n\to\infty}d_{\Omega}(u_n,\sigma_{x_n,z_n}\cup\sigma_{y_n,z_n})=\infty.
\end{align*}

\textbf{Subcase 2:} Now we consider the case where $\hat{x}_\infty\in\partial\widehat{\Omega}$ is finite and $\hat{y}_\infty=\infty$. If $\hat{z}_\infty=\infty$, we can find a $z_n''\in\sigma_{x_n,z_n}$ such that $\sigma_{x_n,z_n''}\subset\mathbb{B}(0,R/2)$ and get a contradiction by the same argument as in Subcase~1. If $\hat{z}_\infty\in\partial\widehat{\Omega}$ is finite, we can find $y_n'\in\sigma_{y_n,z_n}$ such that $\psi_n\circ\sigma_{y'_n,z_n}$ converges to a geodesic in $\widehat{\Omega}$, and use the similar argument in Subcase~1 and we obtain a contradiction again.
\end{proof}

Let $\mathcal{A}$ be the set of geodesic rays from a fixed point $o\in\Omega$. Two geodesics $\sigma_1$, $\sigma_2\in\mathcal{A}$ are equivalent $\sigma_1\sim\sigma_2$ if
\[
\sup_{t\geq0}d_{\Omega}(\sigma_1(t),\sigma_2(t))<\infty
\]
and we define the \textit{Gromov boundary} of $\Omega$ as $\partial^G\Omega:=\mathcal{A}/\sim$. Then we have the following extension theorem:
\begin{thm}
Let $\Omega\subset\mathbb{C}^{d+1}$ be a smoothly bounded pseudo-convex domain of finite type. Suppose the Levi form of every boundary point has rank at least $d-1$. Then the identity map $\Omega\to\Omega$ extends to a homomorphism $\Omega\cup\partial^G\Omega\to\overline{\Omega}$.
\end{thm}

\begin{proof}
By Corollary~\ref{cor4} and Proposition~\ref{prop1} the proof is the same as the proof of \cite[Theorem~1.1.]{Fiacchi2020Gromov} and we omit it.
\end{proof}



\end{document}